\documentclass[12pt]{article}
\usepackage{amssymb}
\usepackage{amsmath}
\usepackage{graphicx}
\newcounter{figurecounter}
\newcounter{figurecounter1}
\newcounter{figurecounter2}
\newtheorem{theorem}{Theorem}
\newtheorem{mtheorem}{Main Theorem}

\newtheorem{corollary}[theorem]{Corollary}

\newtheorem{definition}{Definition}
\newtheorem{example}{Example}
\newtheorem{bexample}{Binary Example}

\newtheorem{lemma}{Lemma}

\newtheorem{proposition}{Proposition}

\newenvironment{proof}[1][Proof]{\textbf{#1.} }{\ \rule{0.5em}{0.5em}}
\newcommand{\calP}{{\cal P}}
\newcommand{\calQ}{{\cal Q}}
\newcommand{\ep}{\varepsilon}

\newcommand{\dR}{{{\bf R}}}
\newcommand{\E}{{{\bf E}}}

\newcommand{\prob}{{{\bf P}}}

\newcommand{\conv}{{{\rm conv}}}

\renewcommand{\baselinestretch}{1.2}

\begin{document}
\renewcommand{\baselinestretch}{1.2}

\title{Strategic Information Exchange}
\author{Dinah Rosenberg\thanks{Economics and Decision Sciences Department, HEC Paris, and GREG-HEC; . e-mail:
rosenberg@hec.fr}\ , Eilon Solan\thanks{The School of Mathematical
Sciences, Tel Aviv University, Tel Aviv 69978, Israel. e-mail:
eilons@post.tau.ac.il} \ and Nicolas Vieille\thanks{Economics and
Decision Sciences Department, HEC Paris, and GREG-HEC. e-mail:
vieille@hec.fr. }}

\maketitle
\date{}

\begin{abstract}
We study a class of two-player repeated games with incomplete
information and informational externalities. In these games, two
states are chosen at the outset, and players get private
information on the pair, before engaging in repeated play. The
payoff of each player only depends on his `own' state and on his
own action. We study to what extent, and how, information can be
exchanged in equilibrium. We prove that provided the private
information of each player is valuable for the other player, the
set of sequential equilibrium payoffs converges to the set of
feasible and individually rational payoffs as players become
patient.
\end{abstract}

Whether and how to acquire information is a question faced by most
decision makers. In statistical decision problems, the decision
maker tries to learn the value of an unknown parameter, and he can
sample from an exogenous population at a fixed cost per draw. In
other contexts, information is held by strategic agents. In
signalling games for instance, a player holding payoff-relevant
private information tries to influence the action choice of an
uninformed party. There, the rationale for disclosing/hiding
information is that the uninformed party's action affects the
payoff of the informed player.  We here study a class of repeated
games with private information, in which there is no such direct
strategic interaction: as in setups of social learning, or of
strategic experimentation, payoffs do not depend upon other
players' actions. However, players hold private information that
has value to other players. Our goal is to understand to what
extent information can be exchanged at equilibrium along the play,
assuming communication is costly. This assumption of pure informational externalities plays
a dual role. Obviously, it simplifies the analysis of the model
and allows to study the exchange of information in isolation from
other strategic considerations. But it leads to a game in which we
might \emph{least} expect exchange of information, and any
positive result in this setup might potentially open the way for
the analysis of other setups.

As an illustration, consider the following game. There are two
biased coins, $C_1$ and $C_2$ (say, with parameter $\frac23$),
which are tossed independently \emph{once} at the outset of the
game. Each of two players, $i=1,2$, has to repeatedly guess the
outcome of coin $C_i$. A correct guess yields a payoff of one,
while an incorrect one yields zero, and successive payoffs are
discounted. If past payoffs are not observed, there is no role for
direct inference, and it is natural to expect that player $i$ will
repeatedly `guess'  the most likely outcome, for an expected
payoff of $\frac23$. Assume that, once coins are drawn, each
player $i$ gets to observe the outcome of coin $C_j$, where $j\neq
i$, but that cheap talk is excluded. This private information has
no `direct' value, but it might have a strategic value, because it
is valuable to the other player. In this game, is there an
equilibrium payoff that improves upon $(\frac23, \frac23)$ ? While
stylized, this game is similar to the situation faced by
executives of two different firms, who hold private information on
their own firm. Trading the stock of one's own firm is illegal, at
times where private information is most valuable, and the
executives may be tempted to implement some implicit collusive
scheme of information exchange through time.

We start with few simple observations. Since cheap talk is assumed
away, exchange of information is to take place through actions, by
`encoding'  privately held information into actions, and by
conditioning the action choice of player $i$ on the outcome of
coin $C_j$. Plainly, there is no equilibrium in which each player
$i$ \emph{fully} `discloses' the outcome of $C_j$ at stage 1.
Indeed, once informed of the outcome of $C_2$, player 2 can make
the correct guess in all subsequent periods, and has no incentive
whatsoever to incur the cost of disclosing information to player
1, including in stage 1. But then, player 1 would not be willing
to play the myopically suboptimal action in stage 1. More
generally, private information is here an asset to be exchanged
for private information, at a cost. On the one hand, a player
cannot disclose information without assigning positive probability
to his myopically suboptimal action, and thereby incurring a
cost.\footnote{Unless if indifferent, but this possibility will
play no role.} On the other hand, a player is not willing to play
a myopically suboptimal action, unless he expects  to be rewarded
with valuable information in return: no player wants to be the
last one to disclose information. Not surprisingly, when the
horizon is finite we prove that $(\frac23,\frac23)$ is the unique
equilibrium payoff. Having an infinite horizon raises the
possibility  of a gradual, open-ended exchange of information.
Since however, the total `amount' of information to be exchanged
is bounded, the feasibility of such a process is not ensured.

Our model is a generalization of the stylized example. Two
`states', $\textbf{s}$ and $\textbf{t}$, are drawn independently,
and players get private information on both $\textbf{s}$ and
$\textbf{t}$.%
\footnote{That is, each player receives information both on his
state and on the other player's state.}
 Next, the players repeatedly choose actions, which are
publicly disclosed. The crucial assumption we make is that a
player's payoff only depends on `his' state, and on his own
action.

We prove that, provided that the information held by each player
is valuable to the other player, the limit set (as $\delta\to 1$)
of sequential equilibrium payoffs coincides with the set of all
feasible payoffs, that are at least equal to the initial, myopic
optimal payoffs. In the simple example discussed above, this limit
set is thus equal to the set $[\frac23,1]\times [\frac23,1]$. Not
only information can be shared, but the rate of information
exchange can be arbitrarily high relative to the discount rate.
Our equilibria share the following features. Players start by
reporting truthfully whatever information they received on their own
state. This leads to a continuation game in which no player holds
private information on his own state. As a result, each player is
able to compute how costly it is for the other player to play his
suboptimal action, and is therefore able to adjust accordingly the
amount of information he discloses as a `reward'. Players next
exchange information in an open-ended manner. The analysis
presents two main and mostly independent difficulties. One is to
design open-ended equilibrium processes, according to which
information is exchanged. In our construction, the bulk of
information exchange takes place early in the game. Later
information disclosure only serves as a means to compensate for
previously incurred costs. The second  consists in adjusting
this continuation play so as to provide the incentives for
truthful reporting of one's information on one's own state.

Our motivation stems from repeated games with incomplete
information. The literature on such games started with Aumann and
Maschler (1966, 1995), and was extensively developed under the
assumption of no discounting, see chapters 5 and 6 in Aumann and
Hart (1992). When there is no discounting, communication through
actions becomes costless, and our model, trivial. Besides the
literature on reputation models, see Mailath and Samuelson (2006)
for a survey, there is only limited work dealing with discounted
repeated games with incomplete information.
Recent contributions are Cripps and Thomas (2003), Peski (2008) and Wiseman (2005).%
\footnote{%
An older one is Mayberry (1967).}
Both Cripps and Thomas (2003), and Peski (2008) look at games with one-sided information,
in which each of the two players knows his own payoff function, and one
of the two is unsure of the payoff function of the other player.
Cripps and Thomas (2003) prove that a Folk Theorem type
of result holds in the limit where the prior belief converges to the case of complete information.
Peski (2008)
essentially shows that all equilibria are payoff-equivalent to equilibria that involve finitely many rounds
of information revelation. Wiseman (2005) looks at situations of \emph{common}
uncertainty. Players share the same information on the underlying state of
nature, and refine this information by observing actual choices and payoffs.

Starting with Crawford and Sobel (1982), the huge literature on
strategic information transmission and on cheap-talk games addresses issues
related to ours. The paper that is closest to our work is Aumann and Hart (2003). There, prior to
playing a game once, two players, one of which is informed of the true game
to be played, exchange messages during countably many periods. Aumann and
Hart (2003) characterize the set of equilibrium payoffs. Following an example of
Forges (1990), they show that allowing for an unbounded communication length
may increase the set of equilibrium payoffs. There are however significant
differences with our setup.
On the one hand, this literature allows the game to exhibit
informational and strategic interaction as well. On the other hand,
information is one-sided, and communication is costless.

Finally, the pattern of information disclosure in our setup is reminiscent
of the pattern of contributions in dynamic models of public good
contributions, see Admati and Perry (1991), Marx and Matthews (2000), or in the
dynamic resolution of the hold-up problem, see Che and Sakovics (2004). More
generally, and as Compte and Jehiel (2004) argue, the existence of a
history-dependent outside option forces equilibrium concessions to be
gradual in many bargaining situations, just as information disclosure has to be
gradual and open-ended here. There are however differences between the results on the two models.
First, in dynamic models of public good contributions
the evolution of contributions follows a deterministic trend,
while in games with incomplete
information beliefs follow a martingale. A more significant difference is the following.
In the former models, there is a one-to-one relation between the contributed amount and
the cost incurred when contributing: the more a player contributes, the
higher his cost by doing so. Here, the cost of disclosing information is
\emph{independent} of the \emph{amount} of information that is disclosed.
The reason is that the cost of disclosure is incurred when playing a
(myopically) suboptimal action, while the amount of information is a
function of how revealing such an action is. The more precise the belief,
the higher the cost of disclosing information.

\bigskip

The paper is organized as follows. Section \ref{sec model}
contains the model and a statement of our main results. Section
\ref{section example} presents the main ideas of the proof through
a version of the example discussed above (allowing for private
signals). Section \ref{sec comments} is devoted to concluding
comments. Proofs are provided in the Appendix.

\section{Model and Main Result}
\label{sec model}

\subsection{The Game}

We study a class of two-player repeated games with incomplete information. At the outset
of the game, a state of the world is realized, and the two players receive private signals,
 $\textbf{l}\in L$ and $\textbf{m}\in M$ respectively. At each stage $n\geq 1$, players choose actions from
the action sets $A$ and $B$. Actions, and only actions, are then publicly disclosed.

All sets are finite. Players share a common discount factor $\delta <1$.

We make the following assumption:
\begin{description}
\item[A.1] The set of states of the world is a product set, $S\times T$, with elements denoted by
$(s,t)$. Player $i$'s payoff depends only on his own action and on the $i$-th component of the state.
That is, the payoff function of player 1 is a function $u:S\times A\to \dR$, while player 2's payoff
is given by a function $v:T\times B\to \dR$.

\item[A.2] Signal sets are also product sets, $L=L_S\times L_T$
and $M=M_S\times M_T$, with elements denoted by $l=(l_S,l_T)$ and
$(m_S,m_T)$. The random triples $(\textbf{s},
\textbf{l}_S,\textbf{m}_S)$ and
$(\textbf{t},\textbf{l}_T,\textbf{m}_T)$ are drawn independently
of each other according to the distributions  $p\in \Delta(S\times
L_S\times M_S)$ and $q\in \Delta(T\times L_T\times M_T)$
respectively.%
\footnote{Here, the state of the world is
$(\textbf{s},\textbf{t})$, the signal to player 1 is
$(\textbf{l}_S,\textbf{l}_T)$, and the signal to player 2 is
$(\textbf{m}_S,\textbf{m}_T)$. We will refer to $\textbf{s}$
(resp., to $\textbf{t}$) as to player 1's state (resp., player 2's
state).}
\end{description}

Assumption \textbf{A.1} ensures that the game is one of pure
informational externalities. Player 1 cares about player 2's
behavior only to the extent that player 2's behavior conveys
information about $\textbf{s}$.

The independence assumption \textbf{A.2} is often made in games
with two-sided incomplete information, see, e.g., Zamir
(1992). Not only does it imply that the two states $\textbf{s}$
and $\textbf{t}$ are independent, but also that the two private
signals, say of player 1, $\textbf{l}_S$ and $\textbf{l}_T$, are
independent. The first component  $\textbf{l}_S$ of player 1's
private signal should be thought of as the information received by
player 1 on his own state $\textbf{s}$, while the second component
$\textbf{l}_T$ is player 1's information on player 2's state,
$\textbf{t}$.\footnote{The subscripts in $l_S$, $l_T$ serve a
mnemonic purpose. We use boldface letters to denote random
variables, when we fear confusion might be at stake.} Besides
allowing for tractability, assumption \textbf{A.2} implies that
behaving myopically is an equilibrium. That is, assume that player
1 repeatedly plays an action $a_\star$ that maximizes the
expectation of $u(\textbf{s},a)$, given  $\textbf{l}_S$. Then, by
\textbf{A.2}, the belief held by player 2 over his own state does
not change along the play, and it is a best reply for player 2 to
repeatedly play an action $b_\star$ that maximizes the expectation
of $v(\textbf{t},b)$, given $\textbf{m}_T$. And vice-versa. If
instead the two states $\textbf{s}$ and $\textbf{t}$ were
correlated, then the choice of $a_\star$ might be informative
about $\textbf{t}$, and may lead to a change in player 2's action,
which would in turn be informative about $\textbf{s}$. But then,
to manipulate this informational feedback, player 1 might be
tempted to `mis-represent' his myopically action $a_\star$. Such
an example is provided in Section \ref{section correlated}.
Assumption \textbf{A.2} assumes away these effects.

In a sense, assumptions \textbf{A.1} and \textbf{A.2} imply that
the only motive for disclosing information on the other player's
state is to get back information in exchange: no information will
ever be exchanged, unless out of purely \emph{strategic} reasons.
We stress that, although the triples $(\textbf{s},
\textbf{l}_S,\textbf{m}_S)$ and
$(\textbf{t},\textbf{l}_T,\textbf{m}_T)$ are independent, we make
no assumption on the distributions $p$ and $q$. Thus, the two
signals relative to a \emph{given} state may be correlated in an
arbitrary way among themselves, or with the state itself.

\bigskip
The main question that we ask is whether and to what extent
valuable information can be exchanged at equilibrium, and how to
organize the exchange of information along the play, while meeting
equilibrium requirements. Our main result consists of a
characterization of the limit set of sequential equilibrium
payoffs, as players become patient.

\bigskip
Strategies will be denoted by $\sigma$ and $\tau$ for players 1
and 2 respectively. A behavior strategy of player $i$ maps his
private information
 and the public history of past moves, into a mixed action. Accordingly, behavior strategies are maps $\sigma
:L\times H\to \Delta(A)$ and $\tau: M\times H\to \Delta(B)$, where
$H=\displaystyle\cup_{n\geq 0}(A\times B)^n$ is the set of finite
sequences of moves. A strategy pair $(\sigma,\tau)$ induces a
probability distribution over the set of (infinite) plays.
Expectations under this distribution are denoted by
$\E_{p,q,\sigma,\tau}$. Thus, the expected discounted payoff of
player 1 is given by
$\gamma_\delta^1(p,q,\sigma,\tau)=\E_{p,q,\sigma,\tau}\left[(1-\delta)\sum_{n=1}^\infty\delta^{n-1}
u(\textbf{s},\textbf{a}_n)\right]$. We denote by $\textbf{p}_n\in
\Delta(S\times M_S)$ the belief held by player 1 at stage $n$: it
is the conditional distribution of the pair
$(\textbf{s},\textbf{m}_S)$, given $\textbf{l}$, and the (public)
sequence of previous moves. Note that $\textbf{p}_1$ depends only
on $\textbf{l}$, because there is no history of moves prior to
stage 1. However, the computation of $\textbf{p}_n$ involves
player 2's strategy for $n>1$.
We denote by $\textbf{q}_n \in \Delta(T \times L_T)$ the belief of player 2 at stage $n$
on his state of nature and on the signal that player 1 received on this state of nature.

\subsection{Preliminaries}

Loosely put, our main result is the following. Provided that each
player holds information that is valuable to the other player, and
that players are patient, any feasible and individually rational
payoff is a sequential equilibrium payoff. Before we state
formally this result, we list some preliminary remarks, and give
few definitions.

\subsubsection{Myopically optimal payoffs}

Given a probability distribution $\pi \in \Delta (S)$ and an action $a\in A$, we write $u(\pi
,a)$ for the expected payoff of player 1 when holding the belief $\pi $ and
playing action $a$:
\[ u(\pi ,a)={\mathbf{E}}_{\pi }[u(\cdot ,a)]=\sum_{s\in
S}\pi (s)u(s,a). \]
We will often abuse notation and write $u(\pi ,a)$ whenever
 $\pi $ is a distribution over a product space of the form $S\times Q$, for some finite set $Q$. In that
case, $u(\pi ,a)=\displaystyle\sum_{s\in S}\pi (\{s\}\times
Q)u(s,a)$.
Given a distribution $\pi \in \Delta(S)$, the \emph{myopically optimal payoff} when holding the belief $\pi $ is
\begin{equation}
u_{\star }(\pi ):=\max_{a\in A}u(\pi ,a).  \label{ustar}
\end{equation}%
 For a given $a\in A$, the map $\pi \mapsto u(\pi ,a)$ is affine. As a
supremum of finitely many affine functions, the map $u_{\star }$ is
convex and piecewise linear.
An action $a\in A$ is \emph{optimal} at $\pi$ if it
 achieves the maximum in (\ref{ustar}).
 \bigskip

 Recall that $\textbf{p}_1\in \Delta(S\times M_S)$ is the (\emph{interim}) belief of player 1 prior to stage 1. Thus, the probability
 $\textbf{p}_1(s,m_S)$ assigned to $(s,m_S)$ is equal to $p(s,m_S|\textbf{l}_S)$. A \emph{myopic} strategy of player 1 is a strategy
 $\sigma_m$ that repeats the same, optimal action at $\textbf{p}_1$. The (\emph{ex ante}) payoff induced by $\sigma_m$ does not depend
 on player 2's strategy, and is equal to
\begin{equation*}
u_\star:={\mathbf{E}}_p\left[u_\star(\mathbf{p}_1)\right]=\sum_{l_S\in L_S}p(l_S)u_\star(p(\cdot|l_S)).
\end{equation*}
If player 2's strategy does not depend on $\textbf{m}_S$, the belief of player 1 does not change
along the play: $\textbf{p}_n=\textbf{p}_1$ for each $n\geq 1$. In such a case, the expected payoff of player 1 does not exceed $u_\star$.
Thus, $u_\star$ is the minmax value for player 1 in the repeated game. For similar reasons,
\[v_\star:={\mathbf{E}}_p\left[v_\star(\mathbf{q}_1)\right]=\sum_{m_T\in M_T}q(m_T)v_\star(q(\cdot|m_T))\]
is the minmax value for player 2.

\bigskip

Player 1's payoff is highest when he knows all he may possibly know about $\textbf{s}$, given the rules of the
game,\footnote{This can be formally deduced from Blackwell Theorem \cite{Bl}.} that is, when player 2's signal $\textbf{m}_S$ is made public. Player 1's
belief is then denoted by
 $\tilde{\textbf{p}}\in \Delta(S)$; thus,
$\tilde{\textbf{p}}(s)=p(s|\textbf{l}_S,\textbf{m}_S)$, for each state $s\in S$. Conditional on both signals $\textbf{l}_S$ and $\textbf{m}_S$, player 1's optimal payoff is $u_\star(\tilde{\textbf{p}})$. Therefore, player 1's \emph{ex ante} expected payoff does not exceed
\begin{equation*}
u_{\star \star }:={\mathbf{E}}_p[u_{\star }(\tilde{\mathbf{p}})]=\sum_{l_S\in L_S,m_S\in M_S}p(l_S,m_S)u_\star(p(\cdot|l_S,m_S)).
\end{equation*}%
Since $u_\star$ is convex, one has $u_{\star\star}\geq u_\star$. This reflects the fact that the marginal value of the information
held by player 2 is nonnegative.

Conversely, any payoff in $[u_\star,u_{\star\star})$ is a feasible payoff for player 1, provided players are patient. Therefore, the (limit)
set of feasible and individually rational payoffs for player 1 is the interval $[u_\star,u_{\star\star}]$.

Similarly, we define $\tilde{\textbf{q}}=q(\cdot|\textbf{l}_T,\textbf{m}_T)$ and
 $v_{\star \star }:={\mathbf{E}}_{q}[v_{\star }(\tilde{\mathbf{q}})]$. As for player 1, the limit set of feasible
 and individually rational payoffs for player 2 is the interval $[v_\star,v_{\star\star}]$.

\bigskip

The example discussed in the introduction will serve as a leading
example. Here, all four sets $S$, $T$, $A$ and $B$ are equal to
$\{0,1\}$. Payoffs are given by $u(s,a)=1$ if $s=a$, and
$u(s,a)=0$ if $s\neq a$ (resp., $v(t,b)=1$ if $t=b$, and
$v(t,b)=0$ if $t\neq b$). We will refer to this setup as the
\textbf{Binary Example}, and we will use it repeatedly, with
various information structures. Note that the myopically optimal
action is $a=1$ if and only if the belief assigned by player 1 to
$s=1$ is at least $1/2$.

\begin{bexample}\label{example-binary}
\label{example} Assume here that $L_S$ is a singleton, while
player 2 observes $\textbf{s}$ (that is, $M_S=S$, and $p(s,m_S)=0$
if $s\neq m_S$). The belief of player 1 can be identified with the
probability $\pi$ assigned to state 1, and
$u_\star(\pi)=\max\{\pi,1-\pi\}$. On the other hand, since player
2 observes $\textbf{s}$, $\tilde{\textbf{p}}$ is either 0 or 1,
and $u_{\star\star}=1$.
\end{bexample}


\subsubsection{Valuable information}

A player will not be willing to play a myopically suboptimal action unless he expects to receive information
in return, the marginal value of which offsets the cost incurred when playing the suboptimal action. In particular, a \emph{necessary}
condition for improving upon $(u_\star,v_\star)$ is that \emph{each} player holds information that is \emph{valuable} to the other player.

We stress that it is not enough that each player holds information on the other player's state. Indeed, if, \emph{e.g.}, the two signals $\textbf{l}_s$ and $\textbf{m}_s$ coincide $p$-a.s., player 1 already knows all player 2 knows relative to $\textbf{s}$. Nor is it enough that players have \emph{private} information, as the next example shows.

\begin{bexample}
\label{example2} Assume that $p(\textbf{s}=1)=\frac34$ and that player 2 receives a binary signal $\textbf{m}_S$ which (conditional on $\textbf{s}$) is correct with probability $\frac23$. Assume moreover that $L_{S}$ is a
singleton, so that player 1 has no information and $\textbf{p}_1=p=\frac34$, and $u_\star=\frac34$.
By Bayes' rule, the posterior probability $\tilde{\textbf{p}}$ assigned to state 1
is equal to $\tilde{%
\mathbf{p}}=\frac{6}{7}$ if $\mathbf{m}_{S}=1$, and is equal
to $\frac{3}{5}$ if $\mathbf{m}_{S}=0$. In either case, the optimal action at $\tilde{\textbf{p}}$ is
$a=1$, and, therefore, knowing the information of player 2 will not change the optimal behavior of player 1:  the information held by player 2 is valueless. Note that
$u_{\star\star}=\E[u (\tilde{\textbf{p}},1)]=u(p,1)=u_\star$, where the middle equality holds
by the law of iterated expectations.
\end{bexample}

In the previous example, the information held by player 2 is valueless
because it does not affect player 1's optimal action. Indeed, the precision of player 2's signal is lower than the prior evidence that the state is 1,
hence player 2's signal cannot provide decisive evidence against state 1.

This observation motivates the definition below.
\begin{definition}
\label{definition valuable} The information of player 2 is \emph{valuable}
for player 1, if
\begin{equation}\label{i-valuable}
{\mathbf{E}}\left[u_{\star }\left(\tilde{\mathbf{p}}\right)|\mathbf{l}_{S}\right]>u_{\star }(%
\mathbf{p}_{1}),\mbox{ with } p-\mbox{probability }1.
\end{equation}
\end{definition}

Similarly, the information of player 1 is valuable for player 2 if ${\mathbf{%
E}}[v_{\star }(\tilde{\mathbf{q}})|\mathbf{m}_{T}]>v_{\star }(\mathbf{q}%
_{1}),$ with $q$-probability 1.

Condition (\ref{i-valuable}) is an interim requirement: it is equivalent to requiring that, after learning $\mathbf{l}%
_{S}$, player 1 assigns positive probability to the event that his
 optimal action would change, if player
2's signal $\textbf{m}_S$ were made public.

This condition implies that $u_{\star\star}>u_\star$, but it is
not implied by this inequality. Indeed, the latter condition is an
\emph{ex ante} requirement, that amounts to requiring that ${\mathbf{E}}\left[u_{\star }\left(\tilde{\mathbf{p}}\right)|\mathbf{l}_{S}\right]>u_{\star }(%
\mathbf{p}_{1})$ holds with positive $p$-probability.

Definition \ref{definition valuable} does not provide a measure of information value, but only a criterium
for deciding whether the information held by a player has positive value or not.
Note that this criterium involves player 1's utility function $u$, and is therefore game-dependent.
\addtocounter{figurecounter}{1}


\subsection{Main result}

Our main result is the following.

\begin{mtheorem}
\label{theorem1}
Assume that the information of each player is valuable to the other player.
Then, as $\delta \rightarrow 1$, the set of sequential equilibrium payoffs
converges to the set $[u_{\star },u_{\star \star }]\times \lbrack v_{\star
},v_{\star \star }]$.
\end{mtheorem}

The proof of Theorem \ref{theorem1} is provided in the Appendix.
Formally we prove that given any compact set $%
W\subset \left( u_{\star },u_{\star \star }\right) \times \left( v_{\star
},v_{\star \star }\right) $, there is $\delta _{0}<1$ such that each vector in
$W$ is a sequential equilibrium payoff of the $\delta $-discounted game, as
soon as $\delta \geq \delta _{0}$.

Note that player 1's expected payoff does not exceed $(1-\delta)u_\star +\delta u_{\star\star}$,
because player 1's action in stage 1 is based only upon $\textbf{l}$.
It will follow from the proof of Theorem \ref{theorem1} that
equilibrium payoffs as high as $u_{\star\star}-(1-\delta)c$
can be implemented in the $\delta$-discounted game,
for some constant $c$.
These two estimates give a rate of convergence of the set of equilibrium payoffs in the $\delta$-discounted game.

The conclusion of the theorem is still valid when players hold different
discount factors, which converge to 1.

\bigskip

A few comments are in order.
The theorem gives a sharp answer to the question we posed.
There exist equilibria, in which almost all information can be exchanged, with a negligible delay,
in spite of the fact that information exchange must be gradual and open-ended. This stands in sharp contrast
to conclusions obtained in dynamic models of public good contribution,
see Compte and Jehiel (2004). The driving force that explains this contrast is the following.
Here, the \emph{cost} of disclosing information is the opportunity cost of playing a suboptimal action,
while the \emph{amount} of information depends on the evolution of beliefs and,
therefore, on the extent to which actions are correlated with private information.
As a result, cost and amount are disentangled. By contrast, in a public good model,
there is a one-to-one relation between the cost of a given monetary contribution, and the amount contributed.
\bigskip

When the information, say, of player 2, has a positive \emph{ex
ante} value ($u_{\star\star}>u_\star$), but is not valuable
according to Definition \ref{definition valuable}, the conclusion
of the theorem does not hold. Indeed, whenever the signal
$\textbf{l}_S$ received by player 1 fails to satisfy inequality
(\ref{i-valuable}), player 1 will infer (at the interim stage)
that the information held by player 2 has no value. Then, player 1
will not be willing to disclose any information to player 2,
although the information he holds may be valuable to player 2. In
Section \ref{sec comments}, we extend the characterization of the
theorem to cover such  cases.
\bigskip

The extension to an arbitrary number of players is outside of the
scope of this paper. With more than two players, there exist cases
where some player $i$ holds no private information, yet receives
information in equilibrium. The basic intuition is that player $i$
may be `rewarded' by some other player $j$, for information that
has been disclosed to $j$ by a third player $k$.

\section{The Binary Example}\label{section example}

We here explain the main ideas of the equilibrium construction
within the binary setup. For simplicity, we assume here that as in
the example in the Introduction, each player knows exactly the
other player's state.\footnote{That is, $\textbf{l}_T=\textbf{t}$
with $q$-probability 1, and $\textbf{m}_S=\textbf{s}$ with
$p$-probability 1.}

All equilibria share the feature that each player starts by
reporting truthfully his private signal relative on his \emph{own}
state, $\textbf{l}_S$ and $\textbf{m}_T$ respectively. The
rationale for this report is the following. In the continuation
game, no player then holds \emph{private} information on his
\emph{own} state. Consequently, player $i$ will always know the
posterior belief of player $j$ on $j$'s state. This in turn allows
player $i$ to compute the perceived cost incurred by player $j$
when playing either action, and to adjust accordingly the amount
of information he discloses. Of course, adequate incentives will
have to be provided to ensure truthful reporting.

We first deal with games in which players get no information
on their own state, which  we call \emph{self-ignorant games}. We
next explain how to provide incentives for a truthful report in
the general case. The discussion of how to combine these two
logically independent steps is relegated to the Appendix.

\subsection{The self-ignorant case}

In this section, we analyze the game discussed in the
introduction. The states $\textbf{s}$ and $\textbf{t}$ are drawn
according to $p \in \Delta(S)$ and $q \in \Delta(T)$. Player 1 is
told $\textbf{t}$, and player 2 is told $\textbf{s}$. Hence, $p$
(and $q$) may be viewed as a distribution over $S=\{0,1\}$. We
identify any distribution over $S$ (resp., over $T$) with the
probability assigned to state $s=1$ (resp., to state $t=1$). For
concreteness, we assume that $p$ and $q$ are such that $p,
q>\frac12$.

\subsubsection{A first equilibrium profile}\label{sec periodic}


We here describe one specific equilibrium profile that will serve
as a building block for the general construction. In this profile,
the play is divided into two phases. In the first phase, at odd
stages (resp. at even stages) player 1 (resp. player 2)
randomizes, thereby transmitting information to player 2 (resp. to
player 1), while player 2 (resp. player 1) plays his myopically
optimal action. The second phase starts when the player who is
randomizing plays his myopically optimal action, and lasts \emph{ad
infinitum}. Along this phase both players play their myopically
optimal action. Thus, the play path looks as follows: in the first
few stages the players alternately play their myopically
suboptimal action, and then the play switches to myopic play.

Randomizations
are informative: in stage 1 for instance, the mixed action used by
player 1 depends on $\textbf{t}$, so that the belief of player 2
in stage 2 depends on player 1's action in stage 1.

The play in the first phase has a cyclic
pattern, with period 4. The evolution of beliefs along the play is
illustrated in Figure \arabic{figurecounter}. This figure involves
two parameters, $p^\star$ and $q^\star$, which will later be
pinned down by equilibrium requirements.

\begin{center}
\label{figurebelief2}
\includegraphics{figurebelief-2.1}\\[0pt]
\end{center}

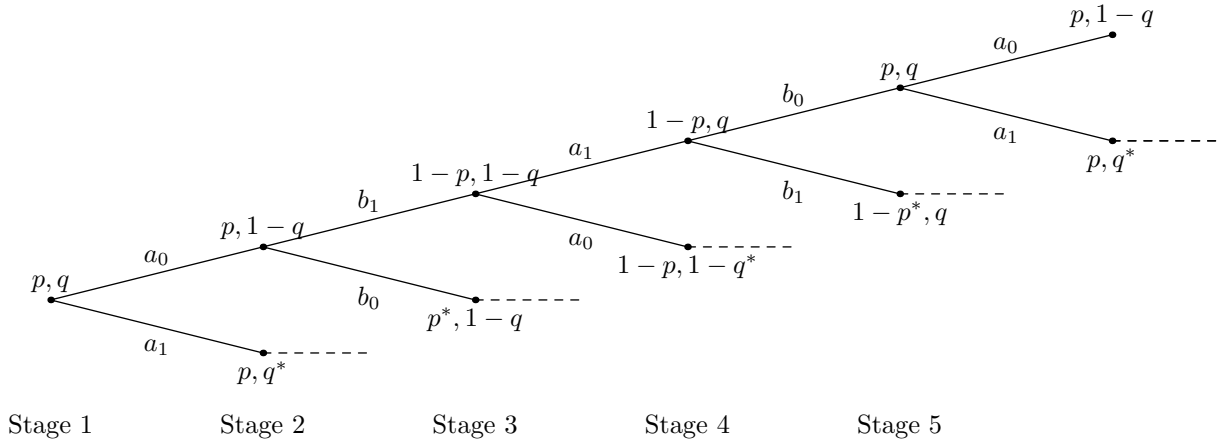
\centerline{Figure \arabic{figurecounter}: The play as long as both players play suboptimally.}

\bigskip

How to read this figure?
Consider first stage 1. In stage 1, player 2 plays with probability 1 his optimal action
(which is $b=1$ since $q>\frac12$). Consequently, the belief of player 1 in stage 2 is $p$,
the same as in stage 1.  Meanwhile, player 1 plays his suboptimal action $a=0$ with probability
$x_\textbf{t}$. The values of $x_0$ and $x_1$ are defined to be
$x_1:=\bar{x}=\displaystyle\frac{1-q}{q}\times \frac{q^\star-q}{q+q^\star-1}$ and
$x_0:=\underline{x}=\displaystyle\frac{q}{1-q}\times \frac{q^\star-q}{q+q^\star-1}$,
where $q^\star$ will be defined below.
By Bayesian updating, the posterior belief $\textbf{q}_2$
of player 2 in stage 2 is equal to $1-q<\frac12$ following $a=0$, and   is equal to $q^\star>q$ following $a=1$.
In player 2's eyes, the suboptimal action $a=0$ is played with probability
$x:= q \bar{x}+(1-q)\underline{x}=\frac{1-\delta}{\delta^2}$.
Note that $x$ and $q^\star$ solve $q= x (1-q)+ (1-x) q^\star$,
which reflects the martingale property of beliefs.
Because $q > 1/2$ it follows that $q^\star > q$ and therefore $x_1 < x_0$.

\addtocounter{figurecounter}{1}

Consider next stage 2. If player 1 played his optimal action $a=1$
in stage 1, players stop exchanging information, and repeat
forever their optimal actions, $a=1$ and $b=1$. If player 1 played
his suboptimal action $a=0$, player 2 reciprocates and assigns a
probability  $y_\textbf{s}$ to his suboptimal action (which is now
$b=1$ because $1-q<\frac12$).  The values of $y_0$ and $y_1$, are
set to $y_1:=\bar{y}=\displaystyle\frac{1-p}{p}\times
\frac{p^\star-p}{p+p^\star-1}$ and
$y_0:=\underline{y}=\displaystyle\frac{p}{1-p}\times\frac{p^\star-p}{p+p^\star-1}$
(where $p^\star$ is defined below) so that the belief of player 1
in stage 3 is either $p^\star$ or $1-p$. The overall probability
assigned to the action $b=1$ is $y:=p \bar{y} + (1-p)
\underline{y}$, which solves $p= y(1-p)+(1-y) p^\star$.

Consider now stage 3 (assuming $a=0$ was played in stage 1). If
the optimal action $b=0$ was played in stage  2, players stop
exchanging information and repeat their optimal actions, $a=1$ and
$b=0$. If player 2 played his suboptimal action in stage 2, player
1 reciprocates and plays as in stage 1, only with the roles of the
states/actions exchanged. To be precise,  player 1 assigns
positive probability to his suboptimal action (which is now $a=1$
because $1-p<\frac12$). As in stage 1, this probability is set to
$\bar{x}$, if the true state $\textbf{t}$ happens to be the state
which  player 2 currently considers more likely (which is now
state $t=0$ because $\textbf{q}_3=1-q<\frac12$), and it is set to
$\underline{x}$ otherwise. By Bayesian updating, the belief of
player 2 in stage 4 is either $q$ or $1-q^\star$. And so on.

Observable deviations\footnote{Such deviations consist in playing the currently suboptimal
action at a node at which the optimal action was expected.} by player $i$ are ignored by player $j$.
\bigskip

The value of $p^\star$ is dictated by equilibrium requirements.
Observe first that, at equilibrium and at any node where player 1 is
supposed to randomize, his expected continuation payoff must be
equal to $u_\star(p)$. Indeed, consider any such node, say at stage
$2n+1$. If he plays his optimal action, player 1 gets an expected
payoff of $u_\star(\textbf{p}_{2n+1})$ in stage $2n+1$, and in all
future stages as well, since the players then stop to exchange
information. Note that $u_\star(\textbf{p}_{2n+1})=u_\star(p)$,
because $\textbf{p}_{2n+1}$ is either equal to $p$ or $1-p$,
depending wether $n$ is even or odd. Because he is randomizing,
player 1 should be indifferent between both actions, and the claim
follows.

Let us now consider stage 1. If player 1 plays his suboptimal
action in stage 1, his continuation payoff from stage 3 on, is
equal to $u_\star(p)$ if player 2 plays $b=1$ in stage 2, and to
$u_\star(p^\star)$ if player 2 plays $b=0$. Since the
probabilities of these two events are equal to $y$ and $1-y$
respectively, the overall payoff of player 1 is then equal to
 \[(1-\delta)(1-u_\star(p))+\delta(1-\delta)u_\star(p)+\delta^2\left(y u_\star(p)+(1-y)u_\star(p^\star)\right),\]
 where the first two terms are the contributions of the first two stages to the overall payoff.
When equating this last expression with $u_\star(p)$, and using
$p>\frac12$, one obtains
\[p^\star=p+(2p-1)\frac{1-\delta}{\delta^2+\delta-1}.\]
The same argument, when applied to player 2, yields $\displaystyle q^\star=q+(2q-1)\frac{1-\delta}{\delta^2+\delta-1}$.

The parameter values $p^\star,
q^\star,\bar{x},\underline{x},\bar{y},\underline{y}$ all lie in
$(0,1)$ as soon as $\ep_\delta\leq p,q\leq 1-\ep_\delta$, where
$\ep_\delta=\frac{1-\delta}{\delta^2+\delta-1}$; that is, as soon
as initial beliefs are not too precise.
Note that $\ep_\delta$ goes to 0 as $\delta$ goes to 1.

Conversely, and whenever $\ep_\delta\leq p,q\leq 1-\ep_\delta$,
this profile is a Nash equilibrium of the repeated game. Observe
that, while player 1's payoff is $u_\star(p)$,  the payoff of
player 2 is equal to
\begin{eqnarray*}
f(q)&:=&(1-\delta)v_\star(q)+\delta \left( x v_\star(q) +(1-x) v_\star(q^\star) \right)\\
&=& v_\star(q)+ \frac{1-\delta}{\delta}\left( 2v_\star(q)-1\right)>v_\star(q).
\end{eqnarray*}

\subsubsection{Further equilibrium payoffs}\label{label further}

We here build upon the previous section, and introduce a class of simple equilibrium
profiles, that implement all equilibrium payoffs (in the limit $\delta\to 1$).


In these equilibria, most of the information on player 2's state
that player 1 will ever transmit is transmitted at stage 1 by
randomizing in that stage, and similarly, most of the information
on player 1's state that player 2 will ever transmit is
transmitted in stage 2. From stage 3
on, the players either implement the equilibrium profile defined
in the previous section, or switch to a myopic behavior. The
equilibrium that is implemented from stage 3 on depends on the
cost of information transmission in stages 1 and 2.

We will now be more precise. Let  $\bar{q},\underline{q},
\bar{p}^0,\underline{p}^0,\bar{p}^1$ and $\underline{p}^1$ be
arbitrary beliefs in $(0,1)$, such that $0 < \underline{q}\leq
\frac12<q <\bar{q} <1$, and $0<\underline{p}^a\leq \frac12 <p<
\bar{p}^a<1$ for $a\in\{0,1\}$. A further condition will be
imposed later. We let the discount factor $\delta$ be sufficiently
large such that all six beliefs, which are determined below, lie
in the interval $(\ep_\delta,1-\ep_\delta)$.

Player 1 randomizes in stage 1, and plays his optimal action in
stage 2. Player 2 plays his optimal action in stage 1, and
randomizes in stage 2. We choose these state-dependent
randomizations in such a way that the beliefs of the players
evolve as indicated in Figure \arabic{figurecounter}. Continuation
payoffs from stage 3 also appear on this figure. Observe that the
optimal action of player 2 in stage 2 depends on the action played
by player 1 in stage 1, since $\underline{q}\leq \frac12
<\bar{q}$.

\begin{center}
\label{figure-evol-belief}
\includegraphics{figure-evol-belief.1}\\[0pt]
\end{center}

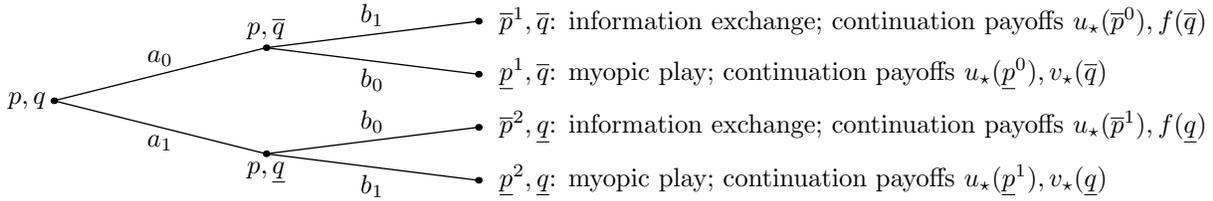
\centerline{Figure \arabic{figurecounter}: The beliefs in the first three stages.}

\bigskip

Since all beliefs lie in $(\ep_\delta,1-\ep_\delta)$, this strategy profile is well-defined. As an example, consider the top final node in Figure \arabic{figurecounter}.
At that node, players' beliefs are $\bar{p}^1$ and $\bar{q}$, and players switch to the equilibrium profile which we designed in the previous section, taking $\bar{p}^1$ and $\bar{q}$ as initial beliefs. \addtocounter{figurecounter}{1}

We claim that equilibrium conditions for player 2 are satisfied. Indeed, note that the function $f(\cdot)$ solves
\[v_\star(\tilde{q})=(1-\delta) (1-v_\star(\tilde{q}))+\delta f(\tilde{q}), \mbox{ for each }\tilde{q}\in [0,1].\]
Hence, whatever be the action played by player 1 in stage 1, player 2 is indifferent between his two actions at stage 2, as desired.

The overall payoff to player 2 is equal to $(1-\delta)v_\star(q)+\delta \left(x v_\star(\bar{q})+(1-x) v_\star(\underline{q})\right).$
As $\delta$ goes to 1, the weight of the first stage decreases to 0. Since $\underline{q}$ and $\bar{q}$ were arbitrary, this  payoff  spans the whole interval $(q,1)= (v_\star,v_{\star\star})$.

\bigskip

It remains to ensure that player 1 is indifferent between both actions in stage 1. That is, the difference in payoffs in stage 1 should be offset by a difference in
the expected continuation payoffs.  This is done by adjusting the amount of information
disclosed by player 2 in stage 2 to the action played by player 1 in stage 1: more information is disclosed if player 1 plays his suboptimal
action in stage 1.

Formally, the indifference condition translates to:
\begin{equation}\label{equilibrium1}
(1-\delta)(u_\star(p)-\left(1-u_\star(p)\right))=
\delta^2(y^0u_\star(\bar{p}^0)+(1-y^0)u_\star(\underline{p}^0)) -
\delta^2(y^1u_\star(\bar{p}^1)+(1-y^1)u_\star(\underline{p}^1)).\end{equation}
A necessary condition is that $\underline{p}^0,\bar{p}^0$ be chosen such that the overall payoff when playing the suboptimal action in stage 1 is
at least $u_\star(p)$. Conversely, one can check that for any such choice of  $\underline{p}^0,\bar{p}^0$, there exist
$\underline{p}^1,\bar{p}^1\in (\underline{p}^0,\bar{p}^0)$ such that equation (\ref{equilibrium1}) holds.

The overall payoff to player 1 is then equal to
\[(1-\delta)(1-u_\star(p))+\delta(1-\delta)u_\star(p)+\delta^2\left(y^0u_\star(\bar{p}^0)+(1-y^0)u_\star(\underline{p}^0)\right).\]
A $\delta$ goes to 1, the weight of the first two stages decreases to zero. Since $\underline{p}^0,\bar{p}^0$ are arbitrary (subject to the individual rationality condition), this payoff spans the whole interval $(p,1)=(u_\star,v_{\star\star})$.



\subsection{The self-informed case}

While sticking to the binary setup, we now allow for private
signals on one's own state, and we informally describe the main
ideas of the construction. All details appear in the Appendix. All
equilibria share a common structure. Equilibrium play is divided
into four successive phases.

In phase 1, players 1 and 2 report their signals $\textbf{l}_S$
and $\textbf{m}_T$, by means of encoding them into finite strings
of actions in a one-to-one way.

The crucial phase is Phase 2. It is designed so as to provide
incentives for truthfully reporting in phase 1. In expectation,
very little information is exchanged in that phase. It is
organized as follows. Each player $i$ draws a random
message,\footnote{We distinguish between \emph{signals}, drawn by
nature, and \emph{messages}, chosen by the players and possibly
subject to strategic considerations.} and `sends' it to player $j$
(again, by encoding messages into sequences of actions). This
message is slightly correlated with player $j$'s state, but is
independent of the report of player $j$ in phase 1.\footnote{Even
if drawn independently of player $j$'s report, it is crucial that
it is sent only after player $j$'s report.} Next, player $j$ plays
a long and deterministic sequence of actions. We refer to these
two subphases as phases 2.1 and 2.2 respectively. The length of
phase 2 is of the order of $\ln (1-\theta)/\ln\delta$ for some
positive $\theta$, hence  phase 2
 contributes a fraction of $\theta$ to the total discounted payoff.\footnote{The value of $\theta$ is
 independent of $\delta$, but may  have to be adjusted to the equilibrium payoff.} The prescribed sequence of actions depends both
 on player $j$'s report in phase 1, and on player $i$'s message in phase 2. The intuition behind this structure is discussed later.

In the continuation game that starts after phase 2, players hold
no private information about their own state and they implement an
equilibrium of the resulting self-ignorant game. The bulk of
valuable information exchange takes place in phase 3. Each player
$i$ draws a random message that is correlated with player $j$'s
state, and `sends' it to player $j$. The degree of correlation is
adjusted as a function of player $j$'s equilibrium  payoff. The
(conditional) law of player $i$'s message does not depend on
player $j$'s report.

Phase 4 consists of the remaining stages. Prior to the first
stage, $N$, of phase 4, each player $i$ assesses the belief held
by player $j$, assuming reports in phase 1 were truthful. He also
assesses the total discounted cost incurred by player $j$ in all
previous stages -- with the exception of the stages of phase 2.2.
Players switch to the equilibrium profile, as designed in the
previous section, which is associated with a payoff  equal to the
myopic payoff, $(u_\star(\textbf{p}_N)$, $v_\star(\textbf{q}_N))$,
plus a \emph{bonus} that exactly offsets the cost incurred
earlier.
 Because few stages are taken into account for the computation of the past cost, this bonus is at most of the order of $(1-\delta)c$, for some positive $c$. Because it is small, hardly any valuable information is exchanged during phase 4. The role of the bonus is to ensure that each player, when randomizing in either phase 2 or phase 3, is indifferent between all messages.

Observable deviations of player $i$ trigger a myopic play of player $j$.
\bigskip

We now provide a few insights into phase 2, by means of two examples.
Phase 2 is designed so that the expected payoff of a player in phase 2.2 is strictly higher
when reporting truthfully than when not.

Assume first that $L_S=\{\underline{l}_S,\bar l_S\}$, and that
player 1 assigns to state $s=1$ a posterior probability $\frac13$
when observing $\underline{l}_S$, and  $\frac23$ when observing
$\bar{l}_S$. Note that the optimal action of player 1 is $a=0$ in
the former case, and $a=1$ in the latter case. Here, incentives
can be provided by simply requiring that player 1 repeats the
action which is optimal given the signal he reported. Indeed,
assume for concreteness that player 1 receives
$\textbf{l}_S=\underline{l}_S$. If player 1 reports truthfully
$\underline{l}_S$ and plays as required the action $a=0$, his
expected payoff is $\frac23$ in all stages of phase 2.2. If
instead player 1 untruthfully claims that he received $\bar{l}_S$ and next
plays the action $a=1$,
which is optimal given his report but in fact suboptimal, his
expected payoff is only $\frac13$. Such an incentive scheme is
appropriate whenever the myopically optimal action of player 1 is
in one-to-one relation with $l_S$.

In some cases, more complex schemes are needed.
Assume that $L_S=\{\underline{l}_S,l_S^\star,\bar{l}_S\}$,
and that the posterior probability assigned to state $s=1$ is respectively $\frac13$ following $\underline{l}_S$,
$\frac12$ following $l^\star_S$ and $\frac23$ following $\bar{l}_S$.
Here, conditional on $\textbf{s}$,
we let player 2 send a message $\mu$ in $\{\ast,0,1\}$. The probability of $\mu=\ast$ is close to 1,
and does not depend on $\textbf{s}$, so that this message conveys no information. Conditional on $\mu\neq \ast$, the
message $\mu$ coincides with $\textbf{s}$ with probability $1-\ep$,
where $\ep>0$ is small. Next, player 1 is to repeat a sequence $\vec a$ of actions of length 8,
which depends on both report and disclosed state:
\begin{itemize}
\item If $\mu=\ast$, the sequence $\vec a$ is independent of player 1's report;
\item Following a report of $l^\star_S$, the sequence $\vec a$
contains four 0's (and four 1's), irrespective of the message
sent by player 2;
\item Following a report of $\underline{l}_S$,
$\vec a$ contains six 0's (and two 1's) if the player 2's message
is $\mu=0$, and seven  0's  if  $\mu=1$;
\item Following a report
of $\bar{l}_S$, $\vec a$ contains six 1's if $\mu=1$, and seven
1's  $\mu=0$.
    \end{itemize}

It is straightforward that strict incentive requirements are met for $\ep>0$ small enough.

\bigskip
We conclude by explaining the equilibrium logic. Assuming that
player $i$ chooses to truthfully report in phase 1, the
computation of the continuation profile in phase 4 ensures that
player $i$ is indifferent between all messages in phases 2.1 and
3, and has no profitable deviation in phase 4. However, the
prescribed sequence of actions in phase 2.2  may involve
suboptimal actions,\footnote{The cost of which is \emph{not} taken
into account in the bonus.} and player $i$ may consider deviating
from this sequence. Such a deviation fails to be profitable as
soon as the marginal value of the information received in phase 3
compensates the cost incurred along phase 2.2. This condition puts
a (mild) constraint on the relation between the duration of phase
2 and the amount of information disclosed in phase 3.

What if player $1$, say, chooses instead to mis-report his signal
in phase 1, and claims that he received the signal $l'_S$ when in
fact receiving $l_S$? If $l_S$ and $l'_S$ are \emph{equivalent},
in the sense that the belief of player 1 is the same following
either signal: $p(\cdot|l_S)=p(\cdot|l'_S)\in \Delta(S)$, such a
deviation is irrelevant.\footnote{We note that the definition of
equivalent signals adopted here is quite specific to the binary
case, and will be different in the general case.} Hence, assume
that player 1 reports a signal  $l'_S$ that is not  equivalent to
$l_S$, and plays consistently with his report
thereafter.\footnote{The case where player $1$ fails to play
consistently, and deviates from the prescribed sequence will pose
no specific problem. It is left to the appendix.} Three properties
combine to ensure that such a deviation fails to be profitable.
First, observe that the joint distribution of the two messages
received from  player $2$ (in phases 2 and 3) does not depend on
player $1$'s report. Hence, misreporting does not affect the
information received from player 2 in phases 2 and 3. Next,
observe that the continuation strategy of player $2$ in phase 4
entails almost no information disclosure. As we show in the
Appendix, this implies that the best-reply continuation payoff of
player $1$ in phase 4, is at most
$u_\star(\textbf{p}_N)+(1-\delta)C$, for some positive $C$, where
$\textbf{p}_N$ is the `true' belief of player 1. That is, in spite
of the fact that player 2's continuation strategy is based on a
\emph{wrong }assessment of player 1's belief,\footnote{Because it
is assuming that player 1 reported truthfully.} the maximal gain
for player 1 is very small. Finally, by design, the expected
payoff received by  player 1 in phase 2 is \emph{strictly} higher
when reporting $l_S$ than $l'_S$.

Because the weight of phase 2 in the total payoff is $\theta$, the small gain obtained in phase 4 can not offset the loss incurred in phase 2.2 when reporting $l'_S$ rather than $l_S$.

\bigskip
\section{Extensions and concluding comments}

\label{sec comments}

\subsection{\emph{Ex ante} vs. \emph{interim} valuable information}

We here discuss how to adapt the statement of the Main Theorem to a situation in which information is \emph{ex ante}
valuable, but need not be \emph{interim} valuable. That is, we assume that both inequalities $u_\star <u_{\star\star}$ and $
v_\star<v_{\star\star}$ hold, but that the information held by player $i$ need not be valuable in the sense of Definition \ref{definition valuable}.

We will assume that with $p$-probability 1, player 1 has a unique
myopically optimal action at $\textbf{p}_1$, and that the
symmetric property.\footnote{If a player has two myopically optimal actions at $\textbf{p}_1$,
he can costlessly reveal information to the other player.} holds for player 2
 Define
 $\bar{L}_S\subset L_S$
to be the set of signals $l_S$ for which the inequality (\ref{i-valuable}) does not hold, and define $\bar{M}_T\subset M_T$ in a symmetric way.
We argue that the limit set of sequential equilibrium payoffs is equal to the set
\[ E:=[u_\star, u_\star q(\bar{M}_T) +u_{\star\star}q(M_T\setminus \bar{M}_T)]\times [v_\star, v_\star p(\bar{L}_S)
 +v_{\star\star}p(L_S\setminus \bar{L}_S)].\]
Observe first that at any equilibrium,  if $\textbf{m}_T\in \bar{M}_T$, player 2 realizes (interim) that the information held by player 1 has no value, and player 2 will then repeat his unique, optimal action at $\textbf{q}_1$. Thus, conditional on $\textbf{m}_T$, and using the independence assumption \textbf{A.2},
player 1's expected payoff is at most $u_\star$ if $\textbf{m}_T\in \bar{M}_T$, and at most  $u_{\star\star} $ if $\textbf{m}_T\notin \bar{M}_T$. Using the same argument for player 2, this shows that any equilibrium payoff vector lies in the set $E$.

Conversely, consider the following class of strategy profiles. In the first two stages, each player $i$ `tells' player
$j$ whether the information held by $j$ has positive value to $i$ or not. This is done as follows.
In stage 1, player $i$ plays his myopically optimal action.
In stage 2, player 1, say, repeats this action if $\textbf{l}_S\in \bar{L}_S$, and switches to a different (suboptimal) action
if $\textbf{l}_S\notin \bar{L}_S$ to signal his willingness to disclose/acquire information. If \emph{both} players switched in stage 2, they implement from stage 3 on  an equilibrium such as we designed in the proof of the Main Theorem.
Otherwise, players repeat their stage 1 action. The sole role of stage 1 is to instruct the other player \emph{how} to interpret the action played in stage 2.

If $\textbf{l}_s\in \bar{L}_S$, it is strictly dominant for player 1 to repeat his optimal action throughout, as required. Indeed,
playing a different action in stage 2 would only lower player 1's payoff, with no benefit since player 2's information is valueless.

If  $\textbf{l}_s\in L_S\setminus \bar{L}_S$,
 player 1's overall payoff is $u_\star(\textbf{p}_1)$ if he pretends that the information
held by player 2 is valueless. However, because there is a
positive $q$-probability that $\textbf{m}_T\notin \bar{M}_T$, it
is a best reply for player 1 to switch to a suboptimal action in
stage 2 as soon as the value of the information disclosed by
player 2 exceeds on average the cost incurred in stage
2.\footnote{It cannot be optimal for player 1 to pretend that
$\textbf{l}_S\notin \bar{L}_S$, yet to lie about his optimal
action.}

In such an equilibrium, conditional on $\textbf{l}_S$, player 1's payoff is  $u_\star(\textbf{p}_1)$, which is then
also equal to $\E[u_\star(\tilde{\textbf{p}})|\textbf{l}_S]$, if $\textbf{l}_S\in \bar{L}_S$.
If instead $\textbf{l}_S\notin \bar{L}_S$, then with probability $q(M_T\setminus \bar{M}_T)$, player 1's payoff may be as high as $\E[u_\star(\tilde{\textbf{p}})|\textbf{l}_S]$. Otherwise, player 1's payoff will be (approximately) $u_\star(\textbf{p}_1)$.
The \emph{ex ante} expected payoff can therefore be as high as
\[u_\star q(\bar{M}_T)  +q(M_T\setminus \bar{M}_T) u_{\star\star}.\]

\subsection{The correlated case}
\label{section correlated}

The case where the independence assumption \textbf{A.2} does not hold raises significant challenges. For the sake of simplicity,
we focus here on the binary setup.

A first difficulty is that myopic play need not be an equilibrium. Hence it is not clear what is the lowest equilibrium payoff.

\begin{bexample}
Here, states and signals are functions of an auxiliary variable $\omega$,
which can assume four different possible values $\omega_i$, $i\in\{1,2,3,4\}$.
The probabilities of the different values, and the way $\omega$ determines the states and signals is as follows:

\[
\begin{array}{c|c|c|c|c|c}
& & \hbox{state of}& \hbox{state of} & & \\
\hbox{state of} & & \hbox{nature $s$ of} & \hbox{nature $t$ of} & \hbox{signal $l$ of} & \hbox{signal $m$ of}\\
\hbox{the world}& \hbox{prob.} & \hbox{player 1} & \hbox{player 2}& \hbox{player 1} & \hbox{player 2}\\
\hline
\omega_1 & \frac16 & 0 & 0 & l^1 & m^1\\
\omega_2 & \frac13 & 0 & 0 & l^2 & m^1\\
\omega_3 & \frac13 & 1 & 0 & l^1 & m^2\\
\omega_4 & \frac16 & 1 & 1 & l^2 & m^2
\end{array}
\]

When behaving in a myopically optimal way, players play as follows. In stage 1, player 2 plays $b=0$.
Indeed, the probability he assigns to state $t=0$ is either 1 or $\frac23$.
Meanwhile, player 1 plays  $a=1$ if $l=l^1$,
and $a=0$ if $l=l^2$.
Indeed, the probability assigned to state $s=0$ is $\frac13$ in the former case,
and $\frac23$ in the latter case. In stage 2,
player 1 repeats his stage 1 action.
On the other hand, player 2 can deduce $\omega$ from player 1's stage 1 choice.
If player 1 played $a=1$, player 2 repeats his stage 1 action.
If instead player 1 played $a=0$, player 2 will play either $b=0$ or $b=1$,
depending on player 2's information. In the former case, players repeat forever their stage 2 action.
In the latter, player 1 deduces $\omega$ from player 2 stage 2's action,
and obtains a payoff of 1 in all later stages.
This obviously creates an incentive for player 1 to deviate in stage 1,
and to always play $a=0$, in order to learn the value of $\omega$ in stage 3.
Thus, it is not an equilibrium to behave myopically. In this example,
information exchange is a \emph{consequence} of equilibrium behavior.
\end{bexample}

A second difficulty lies in understanding the extent to which information can be exchanged.
In Section \ref{sec periodic} we described specific periodic equilibria.
In fact, there are many degrees of freedom in this construction, and the choice of using periodic equilibria was made to facilitate the computation of the equilibrium requirements.
It turns out that when information is correlated there are no degrees of freedom in determining the beliefs,
and the sequence of beliefs is not periodic:
it satisfies uninspiring recursive equations, that do not seem to have closed-form solutions.
Preliminary numerical evidence seems to suggest that some equilibria can be constructed along these lines.

\subsection{The finite horizon case}

We here comment on the assertion that the exchange of information relies on the horizon being unbounded.
We argue that, when the horizon is finite, the myopic equilibrium is the unique Nash equilibrium
of the binary example, as soon as $\textbf{p}_1\neq\frac12$ $p$-a.s., and $\textbf{q}_1\neq \frac12$ $q$-a.s.
This claim is an illustration that in many cases of interest, a
bounded horizon prevents players from exchanging information.

Let
$T$ be the length of the game, and let a Nash equilibrium be
given. We prove that, for any sequence $h_n$ of length $n\geq 0$
of moves that occur with positive probability, the
continuation payoff of player 1 is $u_\star(p_n)$ if $q_n\neq
\frac12$, and the continuation payoff of player 2 is
$v_\star(q_n)$ if $p_n\neq \frac12$. We argue by backward
induction over $n$. The claim holds trivially for  $n=T$. Assume that the claim holds for every
history of length $n+1$, and consider a history of length $n$ such
that $q_n\neq \frac12$. We proceed in two steps.

\bigskip\noindent

\textbf{Step 1}: We claim that $\E[u_\star(\textbf{p}_{n+1})|h_n]=u_\star(p_n)$.

Note first that the claim trivially holds if player 2 does not randomize following $h_n$,
since one then has $\textbf{p}_{n+1}=p_n$ with probability 1.

Assume now that player 2
randomizes at stage $n$ following $h_n$. Observe that $\textbf{p}_{n+1}$ must then be equal to $\frac12$
with positive probability. Otherwise, by the induction hypothesis the continuation payoff of player 2 would be
$v_\star(q_{n+1})$, irrespective of the action played by player 2
in stage $n$. But then player 2 would not be
indifferent in stage $n$ between the two actions -- a
contradiction.

It follows that both possible values of $\textbf{p}_{n+1}$
lie in the same half of the interval $[0,1]$. The claim follows,
because $\E[\textbf{p}_{n+1}|h_n]=p_n$ and because $u_\star$ is linear in both $[0,\frac12]$ and $[\frac12,1]$.

\bigskip\noindent

\textbf{Step 2}: Conclusion

Since $q_n\neq \frac12$, the event $\textbf{q}_{n+1}\neq \frac12$ has a positive
probability conditional on $h_n$. Let $a\in A$ be any action that is played with positive probability
at $h_n$, and following which $\textbf{q}_{n+1}\neq \frac12$.
When playing $a$ and by the induction hypothesis, the continuation payoff of player 1 in stage $n+1$ is equal to $u_\star(\textbf{p}_{n+1})$.
As a consequence, and when playing $a$, the continuation payoff of player 1 in stage $n$ is equal to $(1-\delta)u(p_n,a)+
\E[u_\star(\textbf{p}_{n+1})|h_n]\leq u_\star(p_n)$, using \textbf{Step 1}.
By the equilibrium property, it follows that
the continuation payoff of player 1 in stage $n$ (following $h_n$) is equal to $u_\star(p_n)$.

\bigskip

We note, though, that there may be cases in which information exchange is possible, even when the game is finitely repeated.
The following example illustrates this point.

\label{example19}
Consider the binary example, and add to each action set one more action, that we denote by $2$,
and that yield payoff $\frac23$ irrespective of the state.
The optimal action of player 1, as a function of the belief assigned to state $s=1$,
is given by Figure \arabic{figurecounter} below, and the structure of player 2's best-response is similar.

\begin{center}
\label{figure-best-action}
\includegraphics{figure-best-action.1}\\[0pt]
Figure \arabic{figurecounter}: The optimal action of player 1
\end{center}
\addtocounter{figurecounter}{1}
\begin{example}

Assume that the two states are equally likely, and that player 1
learns $\textbf{t}$ while
 player 2 learns $\textbf{s}$.
Suppose that in stage 1 player 1 plays $[\frac13(a_0),\frac23(a_1)]$ if $\textbf{t}=0$,
and $[\frac23(a_0),\frac13(a_1)]$ if $\textbf{t}=1$,
and suppose that player 2 plays in an analog way.
Player 1's belief in stage 2 is either $[\frac13(0),\frac23(1)]$ or $[\frac23(0),\frac13(1)]$,
depending on player 2's  action in stage 1. In the former case,
 we let player 1 play either $a_0$ or $a_2$ depending on $\textbf{t}$.
 In the latter case, we let player 1 play either $a_1$ or $a_2$,
 depending on $\textbf{t}$.\footnote{We let player 1 repeat $a_2$ forever if player 2 played $b_2$.}
Let the behavior of player 2 in stage 2 be analog.
Provided $\delta$ is high enough,
this strategy pair is an equilibrium, in which players exchange all information in two stages.
\end{example}

The type of construction presented in Example \ref{example19} is valid only for intermediate values of initial beliefs,
and for self-ignorant games.
Moreover, it is not clear how to use
such profiles as continuation profiles in general games.

\subsection{Signals along the Play}

How do our results change when players receive independent signals, both on their own state and
on the other player's state, along the game?
If player $i$ receives signals on his own state along the play,
then, because he is patient,
he will wait until almost all the information that he can get about is own state is received.
However, if in subsequent stages he keeps on receiving information on his own state,
then his belief changes, even though the changes are small.
These changes have the effect that the other player does not know how to compensate player $i$ for playing suboptimally,
and our construction fails.

If, on the other hand, player $i$ received signals on the other player's state along the game,
then again the players can wait until player $i$ receives almost all the information that he can,
and then the players can implement the equilibrium that we construct.
Thus, our results remain valid in this case.

\appendix

\section{Self-Ignorant Games}\label{sec a-1}

We assume throughout the Appendix that all payoffs are in $[0,1]$.
In this section we assume that no player receives private information on his own state. Equivalently,
both sets $L_S$ and $M_T$ are singletons. For simplicity, we here write $L$ and $M$ instead
of $L_T$ and $M_S$.

Let initial distributions $p\in \Delta(S\times M)$ and $q\in \Delta(T\times L)$ be given. We denote the corresponding game by $\Gamma(p,q)$.
W.l.o.g., we assume that $p(m)>0$ for each $m\in M$.\footnote{And that $q(l)>0$ for each $l\in L$: to avoid useless repetitions, we sometimes
state properties for player 1, with the implicit understanding that analog properties hold for player 2 as well.}

We assume that
the information of each player $i$ is valuable to player $j$. Since $\Gamma(p,q)$ is a self-ignorant game, this
is equivalent to assuming that there is no action $a\in A$ that
is optimal at \emph{all} distributions
 $p_m:=p(\cdot|m)$, $m\in M$. Equivalently, one has
\begin{equation}\label{valueignorant}
u_\star=u_\star(p)<u_{\star\star}=\sum_{m\in M}p(m)u_\star\left(p_m\right).
\end{equation}

By Bayes' rule, the belief $\textbf{p}_n$ of player 1 at stage $n$
is the weighted average of $\{p_m \otimes 1_m, m \in M\}$, where
the weight of $p_m \otimes 1_m$ is equal to the probability that
the signal of player 2 is $m$, given player 1's information at
stage $n$.\footnote{This is a way of stating that, as the play
proceeds, the belief of player 1 on $\textbf{m}$ evolves, but the
distribution of $\textbf{s}$ conditional on $\textbf{m}$ remains
fixed.} Thus, the set of possible values of $\textbf{p}_n$ is
\begin{equation}
\label{equ 87}
\Delta^\dag(S\times M) = \conv\{p_m \otimes 1_m, m \in M\}.
\end{equation}
Because $p(m)>0$, $p$ lies in the (relative) interior of
$\Delta^\dag(S\times M)$, which we denote by
$\overset{\circ}{\Delta}^\dag(S\times M)$. We define $q_l =
q(\cdot | l)$ for $l\in L$, and the set $\Delta^\dag(T\times L)$
is defined in a symmetric way.

It is convenient to allow the initial distribution to vary, to account for the fact that beliefs may change along
the play. Since all beliefs lie in $\Delta^\dag(S\times M)$ and $\Delta^\dag(T\times L)$, we will only
consider initial distributions in these sets. We still denote arbitrary such distributions by $p$ and $q$.
\bigskip

In this section, we prove the two propositions below.

\begin{proposition}\label{prop_1}
Let $p\in \overset{\circ}{\Delta}^\dag(S\times M)$ and $q\in \overset{\circ}{\Delta}^\dag(T\times L)$ be given.
There exists $\ep>0$ and $\bar{\delta}<1$ such that the following holds. For every $\delta\geq \bar{\delta}$, every payoff vector in
$[u_\star(p),u_\star(p)+\ep]\times [v_\star(q),v_\star(q)+\ep]$ is a sequential equilibrium payoff of $\Gamma(p,q)$.
\end{proposition}

Given such $p$ and $q$, a payoff vector $\gamma \in
[u_\star(p),u_\star(p)+\ep]\times [v_\star(q),v_\star(q)+\ep]$,
and a discount factor $\delta$, we will construct an equilibrium
profile $(\sigma_{p,q,\gamma},\tau_{p,q,\gamma})$ in
$\Gamma(p,q)$, with payoff $\gamma$. Proposition \ref{prop_2}
bounds the possible gain of player 1 if player 2 has an incorrect
belief on $p$, provided $\gamma$ is close to the myopically
optimal payoff.


\begin{proposition}\label{prop_2}
Let  $p\in \overset{\circ}{\Delta}^\dag(S\times M)$, $q\in \overset{\circ}{\Delta}^\dag(T\times L)$ and
  $c>0$ be given. There exists a constant $C>0$  with the following property.
  For every discount factor $\delta$, every payoff vector $\gamma$ such that $\|\gamma - (u_\star(p),v_\star(q))\|_\infty\leq (1-\delta)c$,
  and every $p'\in \Delta^\dag(S\times M)$, one has
\[\gamma^1(p',q,\sigma,\tau_{p,q,\gamma})\leq u_\star(p')+(1-\delta)C,\mbox{ for every strategy }\sigma.\]
\end{proposition}

In Propositions \ref{prop_1} and \ref{prop_2}, $\ep$, $C$ and $\bar\delta$ may depend a priori on the choice of $(p,q)$.
We will prove that they can be chosen in such a way that the conclusions hold uniformly throughout some neighborhoods of $p$ and $q$.

\subsection{Notations and Preliminaries}

We here describe the main steps leading to the proof of Propositions \ref{prop_1} and \ref{prop_2}.
Our goal is to mimic the recursive construction of the binary case.
We let  $[p^0,p^1]$ be any segment
in the interior of $\Delta^\dag(S\times M)$ such that $u_\star$ is not affine on the segment $[p^0,p^1]$.
The beliefs $p^0$ and $p^1$ take the role of $p$ and $1-p$ in the binary case.

An optimal action $a$ at $p^0$ is not optimal at $p^1$ (and
vice-versa). Otherwise, $a$ would be optimal
 throughout the segment $[p^0,p^1]$, and then $u_\star$ would coincide with the affine map $u(\cdot,a)$ on that segment.

 For $k=0,1$, we let $a^k\in A$ be an optimal action at $p^k$. We denote by $D^1$ the straight line spanned by $p^0$
 and $p^1$ in $\dR^{S\times M}$, and we denote by $\underline{p}$ and $\bar{p}$ the endpoints of
 the segment $D^1\cap \Delta^\dag(S\times M)$, with the convention of Figure \arabic{figurecounter}.

\begin{center}
\includegraphics{figure-straight-line.1}\\[0pt]
Figure \arabic{figurecounter}
\end{center}

\addtocounter{figurecounter}{1}

Let $\pi \in [p_0,\bar{p}]$, and assume that player 1 receives information that changes his belief from
$p^0$ to either $p^1$ (with probability $y$) or $\pi$ (with probability $1-y$).
So that the martingale property of beliefs holds we must have $p^0=y p^1+(1-y)\pi$.
Assume moreover that from the next stage on player 1 receives his myopically optimal payoff.
The gain of player 1 from the information that is revealed to him relative to his myopically optimal payoff at $p^0$, is then
$h_{p^0}(\pi)=\left(y
u_\star(p^1)+(1-y)u_\star (\pi)\right)-u_\star(p^0)$. Since $u_\star$ is convex,
$h_{p^0}(\pi)\geq 0$ for each $\pi$. Since $u_\star$ is not affine
on the interval $[p^0,p^1]$, one also has $h_{p^0}(\pi)>0$ for
$\pi\in (p^0,\bar{p}]$, see Figure \arabic{figurecounter}. In
addition, $h_{p^0}$ is piecewise affine, and non-decreasing as
$\pi$ moves away from $p^0$ towards $\bar p$.

\begin{center}
\includegraphics{figure-h.2}\\[0pt]
Figure \arabic{figurecounter}
\end{center}

Similarly, define $h_{p^1}:[p^1,\underline{p}]\to \dR^+$ by $h_{p^1}(\pi)=\left(yu_\star(p^0)+(1-y)u_\star(\pi)\right)-u_\star(p^1)$, where
$y$ solves $p^1=yp^0+(1-y)\pi$.

We proceed in a symmetric way with player 2. We let $[q^0,q^1]$ be
an arbitrary segment in the interior of $\Delta^\dag(T\times L)$
such that the restriction of $v_\star$ to the segment $[q^0,q^1]$
is not an affine map. We denote by $D^2$ the straight line in
$\dR^{T\times L}$ spanned by $q^0$ and $q^1$, and by
$\underline{q},\bar{q}$ the endpoints of the segment $D^2\cap
\Delta^\dag(T\times L)$. Finally, we define
$h_{q^0}:[q^0,\bar{q}]\to \dR^+$ and
$h_{q^1}:[q^1,\underline{q}]\to \dR^+$ by adapting the definitions
of $h_{p^0}$ and  $h_{p^1}$.

Given a belief $\pi\in \Delta(S\times M)$,  and an action $a\in A$, the \emph{cost} of $a$ at $\pi$ is defined as the
loss incurred when playing $a$ instead of the optimal action at $\pi$:
\[c(\pi,a):=u_\star(\pi)-u(\pi,a).\]
The cost $c(\pi,b)$, for $\pi\in \Delta(T\times L)$ and $b\in B$
is defined analogously.
\bigskip

The proof of Propositions \ref{prop_1} and \ref{prop_2} relies on Lemmas \ref{lemm_1} and \ref{lemm_2} below.

\begin{lemma}\label{lemm_1}
Let $\delta <1$ be such that $\displaystyle \frac{1-\delta}{\delta}c(p^i,a^j)<\max h_{p^i}$ and
$\displaystyle \frac{1-\delta}{\delta}c(q^i,b^j)<\max h_{q^i}$, for  $i, j=0,1$.
Then the vector $(u_\star(p^0),v_\star(q^0)+\frac{1-\delta}{\delta}c(q^0,b^1))$ is a sequential equilibrium payoff of $\Gamma(p^0,q^0)$.
\end{lemma}

\begin{lemma}\label{lemm_2}
Let $\ep>0$ be such that $\ep<\max h_{p^i}$, and $\ep<\max h_{q^j}$ for $i,j=0,1$. There is  $\bar\delta <1$, such that for every
discount factor $\delta\geq \bar{\delta}$, every payoff in $[u_\star(p^0),v_\star(q^0)+\ep]\times [v_\star(q^0),v_\star(q^0)+\ep]$
is a sequential equilibrium payoff of $\Gamma(p^0,q^0)$.
\end{lemma}

We will prove that the conclusion holds uniformly for all initial distributions $\tilde p^i$, $\tilde q^j$ close to $p^i$ and $q^j$.
To be precise, there exist a neighborhood $V(p^i) $ of $p^i$, and a neighborhood $V(q^j)$ of $q^j$ ($i,j\in\{0,1\}$) such that, for every $\delta\geq \bar\delta$,
and every $\tilde p^i\in V(p^i),\tilde q^j\in V(q^j)$, all vectors in $[u_\star(\tilde p^i), u_\star(\tilde p^i)+\ep]\times
[v_\star(\tilde q^j),v_\star(\tilde q^j)+\ep]$ are sequential equilibrium payoffs of $\Gamma(\tilde p^i,\tilde q^j)$.

\subsection{Proof of Lemma \ref{lemm_1}}

In the construction of Section \ref{section example}, the probabilities $x$ and $y$ assigned to suboptimal actions were pinned down by equilibrium requirements. The construction here is slightly more involved
because the number of signals may be larger than 2.

We let $\delta$ be as stated.
Define first $\bar{p}^1\in [p^0,\bar{p})$  by the condition $\displaystyle h_{p^0}(\bar{p}^1)=\frac{1-\delta}{\delta}c(p^0,a^1)$, and $y^1\in [0,1)$
by the equality $p^0=y^1 p^1+(1-y^1)\bar{p}^1$.
The revelation of information just defined offsets the cost to player 1 of playing
the suboptimal action $a^1$ when the belief if $p^0$.
$\bar p^1$ takes the role of $p^\star$ in the binary case.
For $m\in M$, we set $y^1_m=\displaystyle \frac{p^1(m)}{p^0(m)}y^1$.
Because $p^0$ is in the (relative) interior of $\Delta^\dag(S\times M)$, one has $p^0(m)>0$ for each $m$, and $y^1_m\in (0,1)$.
Observe that
$y^1=\sum_{m\in M}p^0(m)y^1_m$, and that the following Bayesian updating property holds. If player 1's belief is $p^0$, and if player 2 plays two different actions $b$ and $b'$ with respective probabilities $y^1_{\textbf{m}}$ and $1-y^1_{\textbf{m}}$, then following $b$ the posterior belief of player 1 is equal to
$p ^1$, and it is equal to $\bar{p}^1$ following $b'$.

Similarly, we let $\bar{p}^0\in (p^1,\underline{p})$ be defined by $h_{p^1}(\bar{p}^0)=\frac{1-\delta}{\delta}c(p^1,a^0)$,
and we set $y^0_m=\displaystyle \frac{p^0(m)}{p^1(m)}y^0$ for $m \in M$, where  $y^0$
solves $p^1=y^0 p^0+(1-y^0)\bar{p}^0$.

We next exchange the roles of the two players, and proceed in  a slightly asymmetric way. We let $\bar{q}^1\in (q^0,\bar{q})$ be defined by
$h_{q^0}(\bar{q}^1)=\displaystyle \frac{1-\delta}{\delta}c(q^0,b^1)$, we let $x^0$ be defined by $q^0=x^0q^1+(1-x^0)\bar{q}^1$, and we set
$x^0_l=\displaystyle \frac{p^1(l)}{p^0(l)}x^0$ for $l\in L$.

We finally define $\bar{q}^0\in (q^1,\underline{q})$, $x^1\in (0,1)$, and $x^1_l=\displaystyle\frac{q^0(l)}{q^1(l)}x^0$ for $l\in L$ in a similar way.

\bigskip

We are now in a position to define strategies $\sigma_\star$ and $\tau_\star$. As long as players alternate in playing their
suboptimal action, player 1 (resp., player 2)  randomizes in each \emph{odd} (resp., in each \emph{even}) stage, and beliefs evolve cyclically:
\[p^0,q^0 \to p^0,q^1 \to p^1,q^1\to p^1,q^0\to p^0,q^0\to \cdots\]
Along this cycle, player 1 assigns a probability  $x^1_\textbf{l}$
to his suboptimal action, $a^1$, when player 2's belief is $q^0$,
and a probability  $x^0_\textbf{l}$ to his suboptimal action,
$a^0$, when player 2's belief is $q^1$. Analog properties hold for
player 2. This is summarized in Figure
\addtocounter{figurecounter}{1} \arabic{figurecounter} below.
\[
\begin{array}{c||c|c|c|c}
\hbox{Stage} & \hbox{player 1} & \hbox{player 2} & \hbox{belief} & \hbox{Suboptimal action}\\
\hline
\hline
1 \mod\ 4 & [x^1_{\textbf{l}}(a^1),(1-x^1_{\textbf{l}})(a^0)] & b^0 & p^0,q^0 & a^1\\
2 \mod\ 4 & a^0 & [y^1_{\textbf{m}}(b^0),(1-y^1_{\textbf{m}})(b^1)] & p^0,q^1 & b^0\\
3 \mod\ 4 & [x^0_{\textbf{l}}(a^0),(1-x^0_{\textbf{l}})(a^1)] & b^1 & p^1,q^1 & a^0\\
0 \mod\ 4 & a^1 & [y^0_{\textbf{m}}(b^1),(1-y^0_{\textbf{m}})(b^0)] & p^1,q^0 & b^1
\end{array}
\]
\centerline{Figure \arabic{figurecounter}: the first phase of play: information exchange.}
\setcounter{figurecounter1}{\value{figurecounter}}

\addtocounter{figurecounter}{1}
As soon as either player 1 plays his optimal action in some odd stage, or player 2 plays his optimal action in some even stage,
the players switch to myopic play forever, as described in columns 3 and 4 in Figure \arabic{figurecounter}.
Here and later, $o(p)$ (resp. $o(q)$) stands for an optimal action of player $1$ at $p$ (resp., of player 2 at $q$).

\[
\begin{array}{c||c|c|c}
\hbox{First stage in which} \\
\hbox{myopically optimal action is played} & \hbox{new belief} & \hbox{player 1} & \hbox{player 2} \\
\hline
\hline
1 \mod\ 4 & p^0,\bar q^1 & a^0& o(\bar q^1) \\
2 \mod\ 4 & \bar p^1,q^1 & o(\bar p^1) & b^1 \\
3 \mod\ 4 & p^1,\bar q^0 & a^1 & o(\bar q^0) \\
0 \mod\ 4 & \bar p^0,q^0 & o(\bar p^0) & b^0
\end{array}
\]
\centerline{Figure \arabic{figurecounter}: the second phase of play: myopic play.}
\setcounter{figurecounter2}{\value{figurecounter}}

We complete the definition of $(\sigma_\star,\tau_\star)$ by specifying actions and beliefs at information sets that are
ruled out by $(\sigma_\star,\tau_\star)$.
For concreteness, we focus on player 1.
An information set of player 1 contains all histories of the form $(l,h)$,  for a fixed signal $l\in L$,
and a fixed sequence $h\in H$ of moves.
Fix an information set that is reached with probability 0 under $(\sigma_\star,\tau_\star)$. We denote it by $I^1_{l,h}$, with $h\in H$. Write $h=(h',(\bar a,\bar b))$,
so that $h'$ is the longest prefix of $h$, and assume that $I^1_{l,h'}$ is reached with positive probability.

We distinguish two cases. Assume first that the action $\bar b$ has probability zero conditional on $h'$. This is the case
where player 2 deviates in an observable way at $h'$.
We let the belief of player 1 at $I^1_{l,h}$ be equal to the belief held at $I^1_{l,h'}$ -- the deviation by player 2 is
interpreted as being non-informative about $\textbf{m}. $
Assume now  that $\bar b$ is played with positive probability at $h'$. In that case, the belief of player 1 at $I^1_{l,h}$ can  be
computed by Bayes' rule, from the belief held at $I^1_{l,h'}$.

In both cases, we let the belief at all subsequent information sets be equal to the belief at $I^1_{l,h}$,
and we let $\sigma_\star$ repeat forever any  action that is optimal at $I^1_{l,h}$.

Observe that, following any history  in $I^1_{l,h}$, under $\tau_\star$ player 2 repeats forever the same action.\footnote{To be precise,
player 2 plays the same action at $I^1_{m,h}$ and in all subsequent information sets.} Indeed, either the sequence $h$ of actions has probability 0,
or it has positive probability. In the former case, the claim follows from the definition of $\tau_\star$ at zero probability information sets.
In the latter case, this implies that the information set $I^1_{l',h}$ has positive probability, for some $l'\neq l$. Since the support
of player 1's mixed actions in the information phase does not depend on his signal, this implies that $I^1_{l,h}$ must belong to the myopic play phase.
Using this observation,
one can check that beliefs are consistent with the strategy profile $(\sigma_\star,\tau_\star)$. We omit the proof.

Note that the strategy $\sigma_\star$ is sequentially rational at any  $I^1_{l,h}$ that is reached with probability 0.
Indeed, since the belief of player 1 is the same at $I^1_{l,h}$ and at all subsequent information sets, it is a best reply to repeat
any action that is optimal at $I^1_{l,h}$.

\begin{lemma}
The profile $(\sigma_\star,\tau_\star)$ is a sequential equilibrium of $%
\Gamma(p^0,q^0)$, with payoff $(u_\star(p^0), v_\star(q^0)+\frac{1-\delta}{%
\delta}c(q^0,b^1))$.
\end{lemma}

We will use this lemma for various distributions $p^0,q^0$. To
avoid confusion, we will then denote the profile
$(\sigma_\star^{p^0,q^0},\tau_\star^{p^0,q^0})$.

\bigskip

\begin{proof}
Each of the strategies $\sigma_\star$ and $\tau_\star$ can be described by an automaton with 8 states:
four states that implement the periodic play in Figure \arabic{figurecounter1},
and four states that implement the myopic play in Figure \arabic{figurecounter2}.

In addition, transitions between (automaton) states are deterministic and depend only on the public history of moves.
Hence, player $i$ can always compute the current state of player $j$'s automaton. Moreover, as can be verified inductively,
the belief of player $i$ following any public history $h$ of moves only depends on the current state of player $j$'s automaton.

It follows that player $i$ has a best response that can be implemented by an automaton that has the same (or smaller) number of states as
the automaton of player $j$.
The dynamic programming principle may be used to identify such a best response.
Using this principle, it is routine to verify that $\tau_\star$ is a best response against $\sigma_\star$, and vice versa.
Indeed, denoting the 8 states of the automata by
$\Omega = {\large\{}(1,\mathtt{periodic})$, $(2,\mathtt{periodic})$, $(3,\mathtt{periodic})$, $(0,\mathtt{periodic})$,
$(1,\mathtt{myopic})$, $(2,\mathtt{myopic})$, $(3,\mathtt{myopic})$, $(0,\mathtt{myopic}){\large\}}$,
the expected payoff to player 2 starting at any given $\omega \in \Omega$ is:
\[
\begin{array}{lll}
V(1,\mathtt{periodic})=v_\star(q^0)+\frac{1-\delta}{\delta}c(q^0,b^1); & & V(1,\mathtt{myopic})=v_\star(q^0);\\
V(2,\mathtt{periodic})=v_\star(q^1); & & V(2,\mathtt{myopic})=v_\star(\bar{q}^1);\\
V(3,\mathtt{periodic})=v_\star(q^\ast)+\frac{1-\delta}{\delta}c(q^1,b^0); & & V(3,\mathtt{myopic})=v_\star(q^1);\\
V(0, \mathtt{periodic})=v_\star(q^0); & \ \ \ \ \ \ \ \ \ & V(0,\mathtt{myopic})=v_\star(\bar{q}^0).
\end{array}
\]
One may verify that for every $\omega\in \Omega$, $V$ solves
\begin{equation}  \label{dyn-prog}
V(\omega)=\max_{b\in
B}\left\{(1-\delta)r(\omega,b)+\delta\sum_{\omega^{\prime}\in
\Omega}V(\omega,b)[\omega^{\prime}]\right\};
\end{equation}
here $r(\omega,b)$ stands for the expected payoff of player 2 when playing $b$ in the (automaton) state $\omega$, where the expectation is taken
w.r.t. the belief held at state $\omega$.
\end{proof}

\subsection{Proof of Lemma \ref{lemm_2}}

Let a payoff vector $\gamma\in [u_\star(p^0),u_\star(p^0)+\ep]\times [v_\star(q^0),v_\star(q^0)+\ep]$ be given.
For $\delta$ high enough, we will define a sequential equilibrium profile in $\Gamma(p^0,q^0)$ with payoff $\gamma$, using the ideas
in Section \ref{label further}. We need some preparations.

Define $\gamma_s^1$ and $\gamma_o^1$ by the equations
\begin{eqnarray*}
\gamma^1&=& (1-\delta)u_\star(p^0)+ \delta (1-\delta) u_\star(p^0)+\delta^2 \gamma_o^1,\\
\gamma^1&=& (1-\delta) u(p^0,a^1) +\delta(1-\delta)u_\star(p^0)+\delta^2\gamma_s^1.
\end{eqnarray*}
$\gamma_s^1$ (resp. $\gamma_o^1$) are the continuation payoffs of player 1 at stage 2, which ensure that the expected payoff of player 1 is $\gamma^1$, if player 1 plays the myopically suboptimal (resp. optimal) action at stage 1,
and the myopically optimal action\footnote{The letters $s,o$ remind that
$\gamma^1_o$ and $\gamma^1_s$ are continuation payoffs following an optimal and a suboptimal action respectively.}
at stage 2.

Define $\gamma^2_o$ be the equality
\[\gamma^2=(1-\delta)v_\star(q^0)+\delta \gamma^2_o.\]
Because $\gamma^1 > u_\star(p^0) > u(p^0,a^1)$ and
$\gamma^2 > v_\star(q^0)$ it follows that
$\gamma_s^1> \gamma_o^1\geq u_\star(p^0)$ while $\gamma^2_o\geq v_\star(q^0)$.
For $\delta$ high enough, and by definition of $\ep$, one has
\[\gamma_s^1<  u_\star(p^0)+ \max h_{p_0}\mbox{ and } \gamma^2_o<v_\star(q^0)+\max h_{q^0}.\]
Hence, there exist $p_s,p_o\in [p^0,\bar p)$, and $q_o\in [q^0,\bar{q})$ such that
\begin{eqnarray*}
h_{p^0}(p_o)&=&\gamma_o^1-u_\star(p^0),\\
h_{p^0}(p_s)&=&\gamma_s^1-u_\star(p^0),\\
h_{q^0}(q_o)&=&\gamma_o^1-v_\star(q^0).
\end{eqnarray*}
Mimicking  the previous section, we define
\begin{itemize}
\item $y_{o,m}=\displaystyle\frac{p^1(m)}{p^0(m)}y_o$, for $m\in M$, where $y_o$ solves
$p^0=y_op^1+(1-y_o)p_o$.
\item $y_{s,m}=\displaystyle\frac{p^1(m)}{p^0(m)}y_s$ ($m\in M$), where $y_s$ solves
$p^0=y_sp^1+(1-y_s)p_s$.
\item $x_l=\displaystyle \frac{q^1(m)}{q^0(m)} x$ for $l\in L$, where
$x$ solves $q^0=x q^1+(1-x)q_o$.
\end{itemize}
\bigskip

\addtocounter{figurecounter}{1}
We are now in a position to define a profile as follows (see also Figure \arabic{figurecounter}).
\begin{description}
\item[Stage 1:] Player 2 plays $b^0$, while player 1 plays the two actions $a^1$ and $a^0$ with probabilities
$x_\textbf{l}$ and $1-x_\textbf{l}$. By Bayesian updating, following $a^1$ the belief of player 2 in stage 2 is equal to $q^1$, and it is equal to $q_o$ following $a^0$ (while the belief of player 1 is still $p^0$).

\item[Stage 2:] Player 2 randomizes. Following $a^1$, player 2 plays the two actions $b^0$ and $b^1$ with probabilities $y_{s,\textbf{m}}$
and $1-y_{s,\textbf{m}}$ respectively. Following $a^0$, he plays the two actions $b^1$ and $o(q_o)$ with probabilities $y_{o,\textbf{m}}$
and $1-y_{o,\textbf{m}}$ respectively. Meanwhile, player 1 plays  $a^0$. By Bayesian updating, the belief of player 1
is equal to (i) $p^1$ following either $(a^0,b^1)$ or $(a^1,b^0)$, (ii) to  $p_s$ following $(a^0,o(q_o))$ and
(iii) to $p_o$
following $(a^1,b^1)$.
\item[Stage 3 and on:] If player 2 played his optimal action in stage 2, players repeat their optimal action.
The continuation payoff is then  $\left(u_\star(p^1),v_\star(q^1)\right)$ following $(a^1,b^1)$ and is
$\left(u_\star(p^1_s),v_\star(q_o)\right)$ following $(a^0,o(q_o))$. Assume now that player 2 played  $b^1$ in stage 2,
following $a^0$. Beliefs are then $(p_o^1,q_c)$ and
    players switch to the equilibrium profile $(\sigma_\star^{p_o^1,q_o},\tau_\star^{p_o^1,q_o})$ of $\Gamma(p_o^1,q_o)$,
    with payoff $\left(u_\star(p_o^1), v_\star(q_o)+\frac{1-\delta}{\delta}c(q_o,b^1)\right)$. Finally, assume that player 2
    played $b^0$ in stage 2, following $a^1$. Beliefs are then $(p^1,q^1)$,
    and players switch to the profile $(\sigma_\star^{p^1,q^1},\tau_\star^{p^1,q^1})$.
\end{description}

\begin{center}
\label{figure-evol-belief2}
\includegraphics{figure-evol-belief.2}\\[0pt]
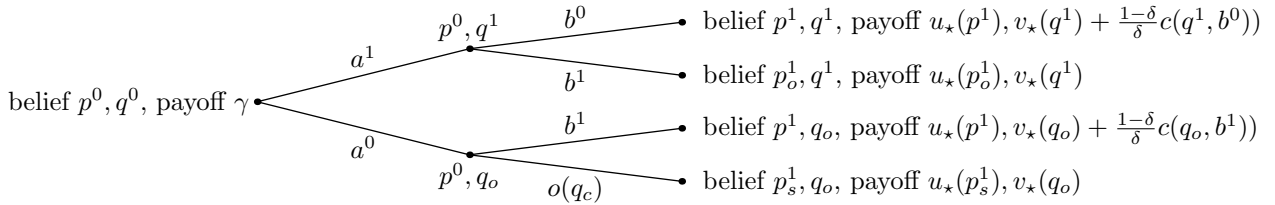
Figure \arabic{figurecounter}: The evolution of beliefs and of continuation payoffs.
\end{center}

Beliefs and actions at information sets that are ruled by this description are defined as in the proof of Lemma \ref{lemm_1}.
The equations defining $\gamma_s^1,\gamma_o^1$ (resp. $\gamma^2_o$) ensure that player 1 is indifferent in stage 1 (resp. player 2 in stage 2) between
the two actions that are assigned positive probability. This implies the equilibrium property.
Details are standard and omitted.

\subsection{Proofs of Propositions \ref{prop_1} and \ref{prop_2}}

We start with the proof of Proposition \ref{prop_1}. The construction we provide here is more complex than needed for Proposition \ref{prop_1}. However, it will facilitate
the proof of Proposition \ref{prop_2}. We let initial distributions $p$ and $q$ be given, in the interiors of
$\Delta^\dag(S\times M)$ and $\Delta^\dag(T\times  L)$. Choose a segment $[p^0,p^1]$ included in the interior
of $\Delta^\dag(S\times M)$, such that (i) $u_\star$ is not affine on $[p^0,p^1]$, and (ii) $p\in (p^0,p^1)$.

By (i) and (ii), one has $u_\star(p)< y u_\star(p^0) +(1-y) u_\star(p^1)$, where $y$ solves $yp^0+(1-y)p^1=p$.
Observe also that the quantity $\tilde y u_\star(p^0) +(1-\tilde y) u_\star(\tilde p^1)$ (with $\tilde y p^0+(1-\tilde y)\tilde p^1=p$) is
strictly decreasing in the neighborhood of
$p^1$, as $\tilde p^1\in [p^0,p^1]$ moves away from $p^1$ and towards $p^0$.

By Lemma \ref{lemm_2}, there exists  $\ep_0>0$,  $\bar \delta <1$,
and neighborhoods $V(p^i)$ and $V(q^j)$ of $p^i$ and $q^j$
($i,j\in\{0,1\}$), such that any payoff in $[u_\star(\tilde
p^i),u_\star(\tilde p^i)+\ep_0]\times [v_\star(\tilde
q^j),v_\star(\tilde q^i)+\ep_0]$ is a sequential equilibrium
payoff of the game $\Gamma (\tilde p^i,\tilde q^j)$, as soon as
$\delta \geq \bar\delta$.

We now prove that the conclusion of Proposition \ref{prop_1} holds with $\ep=\ep_0$. Let $\gamma\in [u_\star(p),u_\star(p)+\ep_0]
\times [v_\star(q),v_\star(q)+\ep_0]$ be given. We describe an equilibrium profile that implements $\gamma$.

One main feature of this profile is the following. As a result of information disclosure by player 2, player 1's belief will move in
\emph{one} stage from $p$ to a belief $\tilde p^i$ close to either $p^0$ or $p^1$. Similarly, player 2's belief will change to a belief
$\tilde q^j$ close to either $q^0$ or $q^1$
in exactly one stage. From that    point on, players implement an equilibrium of $\Gamma(\tilde p^i,\tilde q^j)$ with the
appropriate payoff. There is however one minor difference with previously defined equilibria.
If $u_\star(p) < \gamma^1 < yu_\star(p^0)+(1-y)u_\star(p^1)$,
then the expected payoff of player 1 if we follow the previous construction will be higher than $\gamma^1$, which is the target payoff.
There are two ways to overcome this difficulty.
One way is to choose in this case $p^0$ and $p^1$ which are closer to $p$, thereby lowering
the expected continuation payoff $yu_\star(p^0)+(1-y)u_\star(p^1)$.
A second way, which we adopt here, is to delay information revelation,
so that the discounted payoff is lower than $yu_\star(p^0)+(1-y)u_\star(p^1)$.

\bigskip

Define $N_1\geq 1$ to be the least integer\footnote{$N_1 = \infty$ if $\gamma^1 = u_\star(p)$.} such that $\gamma^1_c\geq y u_\star(p^0)+(1-y)u_\star(p^1)$, where
$\gamma_c^1$ is defined by $\gamma^1=(1-\delta^{N_1})u_\star(p)+\delta^{N_1}\gamma^1_c$.
The inequality $\gamma^1_c\geq y u_\star(p^0)+(1-y)u_\star(p^1)$ ensures that
if player 2 starts revealing information at stage $N_1$,
then one can support $\gamma^1_c$ as a continuation payoff of player 1 at that stage.
Define $N_2$ in a similar way for player 2, and assume w.l.o.g.
that $N_1\leq N_2$. Information is first disclosed at stage $N_1$.
The choice of $N_1$ implies
\[\gamma^1_c -\left(y u_\star(p^0)+(1-y)u_\star(p^1) \right)\leq \frac{1-\delta}{\delta} \left(
y u_\star(p^0)+(1-y)u_\star(p^1) -u_\star(p) \right),\]
provided $\delta$ is high enough.

This implies that for $\delta$ high enough, there is $\tilde{p}^1\in V(p^1)\cap [p^0,p^1]$ such that $\gamma^1_c=
\tilde y u_\star(p^0)+(1-\tilde y) u_\star(\tilde p^1)$, and $\tilde y p^0+(1-\tilde y)\tilde p^1=p$.

\bigskip

We first define a strategy pair $(\sigma,\tau)$ up to stage $N_1+1$. Player 1 repeats an optimal action $o(p)$ at all stages $1,\ldots,N_1$. Player
2 plays $o(q)$ at all stages $1,\ldots, N_1-1$. In stage $N_1$, player 2
plays both actions $o(q)$ and $b'\neq o(q)$ with probabilities such that
 beliefs in stage $N_1+1$ are  $(p^0,q)$ following $b'$, and $(\tilde p^1,q)$ following $o(q)$.

We now define the continuation of $(\sigma,\tau)$ following $o(q)$.
Define $\gamma_c^2$ by the equality $\gamma^2=(1-\delta^{N_1})v_\star(q)+\delta^{N_1}\gamma_c^2$.
The continuation of $(\sigma,\tau)$ in the other case is defined in an analog way,
except that $\gamma^2_c$ has to be replaced by $\displaystyle \gamma^2_c+\frac{1-\delta}{\delta}c(q,b')$,
and the equations that describe equilibrium constraints have to be adjusted.

Let $\tilde N_2$ be the least integer (possibly infinite)
such that $\tilde \gamma^2\geq x v_\star(q^0)+(1-x) v_\star(q^1)$, where $\tilde \gamma^2$ is defined by
$\gamma_c^2=(1-\delta^{\tilde N_2})v_\star(q)+\delta^{\tilde{N}_2} \tilde \gamma^2$.
The choice of $\tilde N_2$ implies
\[\tilde\gamma^2 -\left(x v_\star(q^0)+(1-x)v_\star(q^1) \right)\leq \frac{1-\delta}{\delta} \left(
x u_\star(q^0)+(1-x)u_\star(q^1) -v_\star(q) \right),\]
provided $\delta$ is high enough.
This implies that for $\delta$ high enough, there is $\tilde{q}^1\in V(q^1)\cap [q^0,q^1]$ such that $\tilde \gamma^2=
\tilde x v_\star(q^0)+(1-\tilde x) v_\star(\tilde q^1)$, and $\tilde x q^0+(1-\tilde x)\tilde q^1=q$.

The continuation profile is defined as follows. Player 2 repeats $o(q)$ in all stages $N_1+1,\ldots,N_1+\tilde{N}_2$. Player 1
repeats $o(\tilde p^1)$ in all stages $N_1,\ldots, N_1+\ldots \tilde N_2-1$. In stage $N_1+\tilde N_2$, player 1 plays both actions $o(\tilde p^1)$ and $a\neq o(\tilde p^1)$
with probabilities such that the belief of player 2 is equal to $\tilde q^1$ following $o(\tilde p^1)$, and to $q^0$ following $a$.

Following  $o(\tilde p^1)$, players switch to an equilibrium of the game $\Gamma(\tilde{p}^1,\tilde q^1)$ with payoff
$(u_\star(\tilde p^1),v_\star(\tilde q^1))$. Following  $a$, players switch to an equilibrium of the game $\Gamma (\tilde{p}^1,q^0)$
with payoff $(u_\star(\tilde p^1)+\frac{1-\delta}{\delta}c(\tilde p^1,a),v_\star(q^0)).$

\bigskip
Beliefs and actions off-the-equilibrium-path are defined as in the proof of Lemma  \ref{lemm_1}. The definition of beliefs and continuation payoffs
ensure that players are indifferent whenever randomizing, and that the overall payoff is exactly $\gamma$.

Observe also that there exists a neighborhood $V(p)$ of $p$, with the following property. The two beliefs $p'^{0}$ and $p'^{1}$
associated with $p'\in V(p)$ can be chosen to be continuous in $p'$ and $x'u_\star(p^{'0})+ (1-x')u_\star(p'^{1})$ (with $x'p'^{0}+(1-x')p'^{1}=p'$)
is bounded away from $u_\star(p')$ over $V(p)$. Together with the symmetric property for player 2, this ensures that the robustness result
mentioned after Proposition  \ref{prop_1} holds. This concludes the proof of Proposition
\ref{prop_1}.

\bigskip

We next proceed to the proof of Proposition \ref{prop_2}.
Let $p\in \overset{\circ}{\Delta}^\dag(S\times M)$, $q\in \overset{\circ}{\Delta}^\dag(T\times L)$, and $c>0$ be given. Let $\gamma$
be such that $|\gamma^1-u_\star(p)|\leq (1-\delta)c$ and $|\gamma^2-v_\star(q)|\leq (1-\delta)c$.
Let $[p^0,p^1]$ be the segment associated with $p$ in the proof of Proposition \ref{prop_1},
and let $y$ solve the equation $p=yp^0 + (1-y)p^1$.
Set
\[\eta:=\left(yu_\star(p^0)+(1-y)u_\star(p^1)\right)-u_\star(p)>0,\]
and let $N_1$ be defined as in the proof of Proposition \ref{prop_1}. By construction,
one has
\[(1-\delta^{N_1-1})u_\star(p)+\delta^{N_1-1}(u_\star(p)+\eta)< \gamma^1\leq u_\star(p)+(1-\delta)c,\]
hence $\eta\delta^{N_1-1}\leq (1-\delta)c$. Similarly, one has $\eta\delta^{N_2-1}\leq (1-\delta)c$  (for a possibly lower value of $\eta$).
In the construction of Proposition \ref{prop_1}, players repeat the same action until stage $\min\{N_1,N_2\}$. Therefore, for any $p'\in \Delta^\dag(S\times M)$
and every strategy $\sigma$, one has
\begin{eqnarray*}
\gamma^1(p',q,\sigma,\tau_{p,q,\gamma})&\leq &(1-\delta^{\min\{N_1,N_2\}})u_\star(p')+\delta^{\min\{N_1,N_2\}}\\
&\leq & u_\star(p')+(1-\delta)\frac{\delta c}{\eta}.
\end{eqnarray*}
The result follows, with $C=c/\eta$.


\section{General games}

We here complete the proof of the Main Theorem.
We start with a few notations and remarks in the spirit of Section \ref{sec a-1}. We let initial distributions
$p\in \Delta(S\times L_S\times M_S)$ and $q\in \Delta(T\times L_T\times M_T)$ be given. W.l.o.g., we also assume that $p(l_S)>0$ and $p(l_T)>0$
for each $l_S\in L_S$ and $l_T\in L_T$.\footnote{And we make the symmetric assumption for player 2.}
We assume that the information of
each player $i$ is valuable for the other player. This is equivalent to assuming that, for each $l_S\in L_S$, there is no action $a\in A$ that
is optimal at all beliefs $p_{l_S,m_S}:=p(\cdot|l_S,m_S)$, $m_S\in M_S$.

As the play proceeds, player 2 may disclose information relative to $\textbf{m}_S$, and player 1's belief about $\textbf{m}_S$ may change.
Analogously to the case of self-ignorant games (see Eq. (\ref{equ 87})),
the belief $\textbf{p}_n$ of player 1 given  $\textbf{l}_S=l_s$ is always in the set
\[ \Delta_{{l}_S}^\dag(S\times M_S) = \conv\{p_{l_S,m_S} \otimes 1_{l_S} \otimes 1_{m_S}, m_S \in M_S\}. \]
Note that $p(\cdot|l_S)$ lies in the relative interior of the set $\Delta^\dag(S\times M_S)$.

For $m_T\in M_T$, we define $\Delta^\dag_{m_T}(T\times L_T)$ in a symmetric way. The results of Section \ref{sec a-1}
will be applied to the different sets $\Delta^\dag_{l_S}(S\times M_S)$ and $\Delta^\dag_{m_T}(T\times L_T)$ of initial distributions.
\subsection{Providing Incentives}

For simplicity, we focus here on player 1. Analog properties hold for player 2 as well.
We first define an equivalence relation $\sim$ over
$L_S$. As we will see, two signals $\underline{l}_S$ and $\bar l_S$ such that $\underline{l}_S\sim \bar l_S$  may be merged, and treated as a single signal. Given $l_S\in L_S$, we define a vector $\vec{Z}^{l_S}$ of size
$M_S\times A\times A$ by
\[\vec{Z}^{l_S}_{m_S,a,a'}:=p(m_S|l_S)\left(u(p_{l_S,m_S},a)-u(p_{l_S,m_S},a')\right),\mbox{ for } m_S\in M_S,\mbox{ and }a,a'\in A.\]
Because the information held by player 2 is valuable for player 1, $\vec{Z}^{l_S}\neq \vec{0}$, for each $l_S\in L_S$.\footnote{Indeed, if $\vec{Z}^{l_S}=\vec{0}$, then
\emph{any} action $a\in A$ is optimal at $p_{l_S,m_S}$, for \emph{each} $m_S$.}

\begin{definition}
Let $\underline{l}_{S}, \bar{l}_{S}\in L_S$ be given. The two signals $%
\underline{l}_{S}$ and $\bar{l}_{S}$ are \emph{equivalent},
written $\underline{l}_{S}\sim\bar{l}_{S}$,
if the two vectors $\vec{Z}^{\underline{l}_{S}}$ and $\vec{Z}^{\bar{l}_S} $ are positively collinear, that is, if
\begin{equation}
\label{equ equivalent}
\exists \alpha > 0, \vec Z^{\bar{l}_{S}}=\alpha \vec Z^{\underline{l}_{S}}.
\end{equation}
\end{definition}
Plainly, if the two distributions $p(\cdot|\underline{l}_S)$ and $p(\cdot|\bar{l}_S)$ in $\Delta(S\times M_S)$ coincide, then
$\underline{l}_S\sim\bar{l}_S$. However, the converse implication does not hold.

Observe that, if $\underline l_S\sim\bar l_S$
and $\vec Z^{\bar{l}_{S}}=\alpha \vec Z^{\underline{l}_{S}}$,
then for every two mixed actions $x,x' \in \Delta(A)$ we have:
\begin{equation}
\label{equ 510}
p(m_{S}|\bar{l}_{S})\left(u(p_{\bar{ l}_{S},m_{S}},x)-u\left(p_{\bar{ l}_{S},m_{S}},x'\right)\right)
=\alpha p(m_{S}|\underline l_{S})\left(u(p_{\underline l_{S},m_{S}},x)-u\left(p_{\underline l_{S},m_{S}},x'\right)\right).
\end{equation}

As a preparation for Lemma \ref{lemm equi} below, observe that a strategy $\sigma$
may be viewed as a collection $(\sigma_{l_S})_{l_S\in L_S}$, with the interpretation that $\sigma_{l_S}:L_T\times H\to \Delta(A)$
is the `interim' strategy used if $\textbf{l}_S=l_S$.\footnote{To be formal, $\sigma_{l_S}(l_T,h)$ is defined to be $\sigma(l_S,l_T,h)$.}

\begin{lemma}\label{lemm equi}
Let $\tau$ be any strategy of player 2. Then there exists a best
reply $\sigma$ of player 1 to $\tau$ such that
$\sigma_{\bar{l}_S}=\sigma_{\underline{l}_S}$ whenever
$\bar{l}_S\sim\underline{l}_S$.
\end{lemma}

According to Lemma \ref{lemm equi}, player 1 has a best reply that
depends only on the equivalence class of $\textbf{l}_S$.
\bigskip\noindent

\begin{proof}
Let a strategy $\tau$ of player 2 be fixed throughout. Given $f:H\to \Delta(A)$, and $(l_S,l_T)\in L_S\times L_T$, we denote by $\gamma^1(f,\tau|l_S,l_T)$ the interim expected payoff of player 1, when getting $\textbf{l}_S=l_S$, $\textbf{l}_T=l_T$, and when playing according
to $f$ thereafter. Given $n\geq 1$, we also denote by $g_n^1(f,\tau|l_S,l_T)$ the corresponding payoff at stage $n$.

We let $\bar{l}_S\sim\underline{l}_S$ be any two equivalent signals, so that $\vec{Z}^{\bar{l}_S}=\alpha\vec{Z}^{\underline{l}_S }$
for some $\alpha >0$. We will prove that, for every two ``interim strategies'' $f:H\to \Delta(A)$ and $f':H\to \Delta(A)$, for every $l_T\in L_T$
and every stage $n\geq 1$, one has
\begin{equation}\label{payoff_equi}
g^1_n(f,\tau|\bar l_S,l_T)-g^1_n(f',\tau|\bar l_S,l_T)=\alpha\left( g^1_n(f,\tau|\underline l_S,l_T)-g^1_n(f',\tau|\underline l_S,l_T)\right).
\end{equation}
Equation (\ref{payoff_equi}) will imply that
\[\gamma^1(f,\tau|\bar l_S,l_T)-\gamma^1(f',\tau|\bar l_S,l_T)=\alpha\left( \gamma^1(f,\tau|\underline l_S,l_T)-\gamma^1(f',\tau|\underline l_S,l_T)\right),
\]
from which the result follows.
Indeed, if $f$ is better than $f'$ when the signal is $\bar{l}_S$,
then it is also the case when the signal is $\underline{l}_S$.
Therefore if $f$ is a best response when the signal is $\bar{l}_S$,
then it is also a best response when the signal is $\underline{l}_S$.

We let a stage $n\geq 1$ be given. We fix $(l_S,l_T)\in L_S\times L_T$, and we decompose the payoff $g^1_n(f,\tau|l_S,l_T)$ as follows. For a given sequence of moves $h\in H_n:=(A\times B)^{n-1}$, we denote by $\prob_{f,\tau}(h|l_S,l_T)$
the probability that $h$ occurs, when $(\textbf{l} _S,\textbf{l}_T)=(l_S,l_T)$ and players
play according to $f$ and $\tau$. We denote by $\prob_{f,\tau}(\cdot|h,l_S,l_T)\in \Delta(S\times M_S)$ the belief which is then held by player 1.

With these notations, one has
\begin{equation}\label{payoff_equi2}
g^1_n(f,\tau|l_S,l_T)=\sum_{h\in H_n}\prob_{f,\tau}(h|l_S,l_T)u\left(\prob_{f,\tau}\left(\cdot|h,l_S,l_T\right)),f(h)\right).
\end{equation}

The belief of player 1 following $h$  is given by
\[\prob_{f,\tau}(s|h,l_S,l_T)=\frac{1}{\prob_{f,\tau}(h,l_S,l_T)}\sum_{m_S\in M_S}\prob_{f,\tau}(s,h,l_S,l_T,m_S),\mbox{ }s\in S,\]
where $\prob_{f,\tau}(h,l_S,l_T)=\prob_{f,\tau}(h|l_S,l_T)p(l_S)q(l_T)$. Because the state $\textbf{s}$ and the history of moves until stage $n$ are conditionally independent given $(\textbf{l}_S,\textbf{l}_T,\textbf{m}_S)$, this belief is equal to
\begin{equation}
\label{equ 88}
\prob_{f,\tau}(s|h,l_S,l_T)=\frac{1}{\prob_{f,\tau}(h,l_S,l_T)}\sum_{m_S\in M_S}\prob_{f,\tau}(h|l_S,l_T,m_S)p_{l_S,m_S}(s)p(l_S,m_S)q(l_T).
\end{equation}
Plugging (\ref{equ 88}) into (\ref{payoff_equi2}), and using the linearity of $u$, one gets
\begin{equation}\label{equ_511}
g^1_n(f,\tau|l_S,l_T)=\sum_{h\in H_n}\sum_{m_S\in M_S}p(m_S|l_S)\prob_{f,\tau}(h|l_S,l_T,m_S)u(p_{l_S,m_S},f(h)).\end{equation}

Using (\ref{equ_511}) for both $f$ and $f'$,
and because $\sum_{h \in H_n}\prob_{\sigma,\tau}(h \mid l_S,l_T,m_S) =
\sum_{h' \in H_n} \prob_{\sigma',\tau}(h' \mid l_S,l_T,m_S)=1$,
we obtain:
\begin{small}
\begin{eqnarray}
\label{equ141}
&&g^1_n(f,\tau|{l}_{S},l_{T})-g^1_n(f^{\prime},\tau|{l}_S,l_{T})\\
&&=
\sum_{m_S \in M_S} p(m_S|l_S)\left(\sum_{h \in H_n}
\prob_{f,\tau}(h|l_S,l_T,m_S)  u(p_{l_S,m_S},f(h)) -
\sum_{h' \in H_n}
\prob_{f',\tau}(h'|l_S,l_T,m_S) u(p_{l_S,m_S},f'(h'))\right)\nonumber\\
&&=
\sum_{m_S \in M_S} \sum_{h \in H_n} \sum_{h' \in H_n}
\prob_{f,\tau}(h|l_S,l_T,m_S)
\prob_{f',\tau}(h'|l_S,l_T,m_S) \times p(m_S|l_S)
\left( u(p_{l_S,m_S},f( h)) -
u(p_{l_S,m_S},f'(h'))\right).\nonumber
\end{eqnarray}
\end{small}
Because $\underline l_S$ and $\bar l_S$ are equivalent, Eq. (\ref{payoff_equi}) follows by (\ref{equ 510}),
and (\ref{equ141}) applied to both $l_S = \underline{l}_S$
and  $l_S = \bar{l}_S$.
\end{proof}
\bigskip

Lemma \ref{lemm equi} implies that any equilibrium of the modified game in which player 1 only observes the equivalence class of $\textbf{l}_S$ is an equilibrium of the original game. Put it differently, the set of equilibrium payoffs of the game in which players do not distinguish between equivalent signals is a subset of the set of equilibrium payoffs of the game we started with. Besides, the values of $u_\star$ and $v_\star$ (resp., of
$u_{\star\star}$ and $v_{\star\star}$) are the same for both games.

Therefore, it is sufficient to prove that
 the conclusion of the Main Theorem holds  for the modified game. In particular, we may and
 will assume from here on that, for every two signals $\bar l_S\neq \underline l_S$, the vectors $\vec{Z}^{\bar l_S}$ and
 $\vec{Z}^{\underline l_S}$ are \emph{not} positively collinear. We also make the symmetric assumption for player 2.
A direct consequence of this assumption is Corollary \ref{cor_equi} below.
\begin{corollary}\label{cor_equi}
Let $\bar{a}\in A$ be arbitrary. For $l_S\in L_S$, define the vector $\vec{Y}^{l_S}$ of size $M_S\times A$ by
\[\vec{Y}^{l_S}_{m_S,a}:=p(m_S|l_S)\left(u(p_{l_S,m_S},a)-u(p_{l_S,m_S},\bar a)\right), \ \ \ m_S \in M_S, a \in A.\]
Then for every two signals $\bar l_S\neq \underline l_S$, the two vectors $\vec{Y}^{\bar l_S}$ and $\vec{Y}^{\underline l_S}$ are not positively collinear.
\end{corollary}

The vector $\vec{Y}^{l_S}$ is equal to the projection of $\vec{Z}^{l_S}$ on a lower-dimensional space. Hence, linear independence of $\vec{Y}^{\bar l_S}$ and of $\vec{Y}^{\underline l_S}$ does not follow in general from linear independence of $\vec{Z}^{\bar l_S}$ and $\vec{Z}^{\underline l_S}$, and an \emph{ad hoc} proof is needed.

\begin{proof}
We argue by contradiction, and assume that $\vec{Y}^{\bar l_S}=\alpha \vec{Y}^{\underline l_S}$ for some $\alpha> 0$. Let $m_S\in M_S,a,a'\in A$ be arbitrary.
Observe that $\vec{Z}^{l_S}_{m_S,a,a'}=\vec{Y}^{l_S}_{m_S,a}-\vec{Y}^{l_S}_{m_S,a'}$, for $l_S=\bar l_S, \underline l_S$. Hence $\vec{Z}^{\bar l_S}=
\alpha\vec{Z}^{\underline l_S}$, a contradiction.
\end{proof}

\bigskip
The next lemma is central to the provision of incentives (phase 2 of the equilibrium play). Given $x:L_S\times M_S\to \Delta(A)$,
and for every $l_S,k\in L_S$, we define
\[\E_x[l_S\to k]=\sum_{m_S\in M_S}p(m_S|l_S)u(p_{l_S,m_S},x_{k,m_S}),\]
with the following interpretation. The expression $\E_x[l_S\to k]$ is the expected stage payoff when player 1 gets $\textbf{l}_S=l_S\in L_S$,
`reports' $k\in L_S$, is told $\textbf{m}_S$, and plays the mixed action $x_{k,\textbf{m}_S}$
that depends on player 1's report, and on player 2's signal.\footnote{We use the different letter $k$ to
distinguish between a signal  and a report, although both belong to the same set $L_S$.}
According to Lemma \ref{lemm_incen} below, the map $x$ can be chosen in a way that this expected payoff is highest when reporting truthfully.

\begin{lemma}\label{lemm_incen}
There exists $x^\star:L_S\times M_S\to \Delta(A)$, such that
\[\E_{x^\star}[l_S\to k]<\E_{x^\star}[l_S\to l_S],\mbox{ for every }l_S, k\in L_S, l_S\neq k.\]
\end{lemma}

\begin{proof}
Let $\bar a\in A$ be arbitrary, and let $x:L_S\times M_S\to \Delta(A)$ be given. For $k\in L_S$, we define a vector $\vec{X}^k$ of
size $M_S\times A$ by $\vec{X}^k_{m_S,a}:=x_{k,m_S}(a)$,  $m_S\in M_S,a\in A$. Observe that $x_{k,m_{S}}(\bar{a})=1-%
\displaystyle \sum_{a\neq \bar{a}}x_{k,m_{S}}(a)$. Hence, $\E_x[l_S\to k]$ may be rewritten as
\begin{eqnarray*}
{\mathbf{E}}_{x}[l_{S}\to k]&=&\sum_{m_{S}\in
M_S}p(m_{S}|l_{S})u(p_{l_{S},m_{S}},\bar{a}) \\
&& +\sum_{m_{S}\in M_S}\sum_{a\in
A}p(m_{S}|l_{S})x_{k,m_{S}}(a)\left(u\left(p_{l_{S},m_{S}},a\right)-u%
\left(p_{l_{S},m_{S}},\bar{a}\right)\right)\\
&=& \vec{Y}^{l_S}\cdot \vec{X}^k+\sum_{m_{S}\in
M_S}p(m_{S}|l_{S})u(p_{l_{S},m_{S}},\bar{a}).
\end{eqnarray*}
Because the second term in the last displayed equation does not depend on $k$, it is sufficient to construct $x$ such that
\begin{equation}\label{incentive}
\vec{Y}^{l_S}\cdot \vec{X}^k<\vec{Y}^{l_S}\cdot \vec{X}^{l_S}\mbox{ for every }l_S, k\in L_S, l_S\neq k.\end{equation}

For $l_{S}\in L_S$, define
\begin{equation*}
\tilde{X}^{l_{S}}:=\frac{1}{\|\vec{Y}^{l_{S}}\|_2}\vec{Y}^{l_{S}},
\end{equation*}
and let $l_{S}\neq k$ be arbitrary in $L_S$.
Then by the Cauchy-Schwartz inequality,
\begin{equation*}
\tilde X^{k} \cdot \vec{Y}^{l_S} =
\frac{\vec{Y}^{k}}{\|\vec{Y}^{k}\|_2} \cdot \vec{Y}^{l_S} <
\| \vec{Y}^{l_S} \|_2 =
\frac{\vec{Y}^{l_S}}{\|\vec{Y}^{l_S}\|_2} \cdot \vec{Y}^{l_S} =
\tilde X^{l_S} \cdot \vec{Y}^{l_S},
\end{equation*}
where the strict inequality holds since $\vec{Y}^k$ and $\vec{Y}^{l_S}$ are not positively collinear.
Therefore, (\ref{incentive}) holds with $(\tilde{X}^{l_{S}})_{l_{S}\in L_S}$. Note that (%
\ref{incentive}) still holds when the same constant is added to all
components, and/or when all components are multiplied by the same constant $%
\phi >0$. Choose $\beta\in {\mathbf{R}}$ and $\phi>0$ such that all
components of $\phi \tilde{X}^{l_{S}}+\beta$ lie in $(0,\frac{1}{|M_S\times A|})$%
, for all $l_{S}$.
Because $Y^{l_S}_{m_S,\bar a} = 0$, it suffices to set
\begin{equation*}
x^\star_{l_{S},m_{S}}(a)=\phi \tilde{X}^{l_{S}}_{m_{S},a}+\beta\mbox{ for } a\neq \bar{a},
\end{equation*}
and $x^\star_{l_{S},m_{S}}(\bar{a})=1-\displaystyle \sum_{a\neq \bar{a}}x_{l_{S},m_{S}}(a)$%
.
\end{proof}

Given $\ep_2:M_S\to \Delta(M_S)$, and $l_S,k\in L_S$, we define
\[\E_{\ep_2,x^\star}[l_S\to k]=\sum_{m_S,\mu\in M_S}p(m_S|l_S)\ep_2(\mu|m_S)u(p_{l_S,\mu},x^\star_{k,\mu}).\]
This is the expected  stage payoff of player 1 when (i) player 1 gets $\textbf{l}_S=l_S$, and  `reports' $k$,
(ii) player 2 draws $\mu\in M_S$ according to $\ep_2(\cdot|\textbf{m}_S)$ and (iii) player 1 plays $x^\star_{k,\mu}$.
We here abuse notation and write $p_{l_S,\mu}$ for the belief of player 1, given $l_S$ and $\mu$.\footnote{Note that, for fixed $\mu$,  the
belief $p_{l_S,\mu}$ depends on $\ep_2$, although this is not emphasized in the notation.}

Observe that the expectation $\E_{\ep_2,x^\star}[l_S\to k]$ is continuous w.r.t. $\ep_2$, and that $\E_{\ep_2,x^\star}[l_S\to k]$
is equal to $\E_{x^\star}[l_S\to k]$ when $\ep_2(\cdot|m_S)$ assigns probability 1 to $m_S$, for each $m_S$.
Corollary \ref{cor_incen} below  therefore follows from Lemma \ref{lemm_incen} by continuity.

\begin{corollary}\label{cor_incen}
There exists $\ep_2:M_S\to \overset{\circ}{%
\Delta}(M_S)$, such that
\begin{equation}  \label{misreport2}
{\mathbf{E}}_{\varepsilon_2,x^\star}[l_{S}\to k]< {\mathbf{E}}_{%
\varepsilon_2,x^\star}[l_{S}\to l_{S}], \mbox{ for every }l_S, k\in L_S, l_S\neq k.
\end{equation}
\end{corollary}

We fix $\ep_2$ and $x^\star$ for the rest of the paper. Because the distribution $\ep_2(\cdot|m_S)$ has full
support, the conditional distribution $p_{l_S,\mu}$ lies in the relative interior of $\Delta^\dag_{l_S}(S\times M_S)$ (for each $\mu\in M_S$).
Define $\ep_1$ analogously.

\subsection{Equilibrium strategies -- Structure}

We let a payoff vector $\gamma=(\gamma^1,\gamma^2)$ be given, with $ u_\star<\gamma^1<u_{\star\star}$ and $v_\star<\gamma^2<v_{\star\star}$. We
will construct a sequential equilibrium with payoff $\gamma$. We let the discount factor $\delta$ be given.
In the construction we add one additional message, $\square$, to each player.

Given $x\in \Delta(A)$, and given a number $N$ of stages, we denote by $\vec{a}^N(x)\in A^N$, a sequence of actions of length $N$
that provides the best approximation of the mixed action $x$ in terms of discounted frequencies. That is, $\vec{a}^N(x)=(a_n)_{1\leq n\leq N}$ is
chosen to minimize $\|x_\delta(\vec{a}^N)-x\|_\infty$, where
\[x_\delta(\vec{a}^N)[a]:=\frac{1-\delta}{1-\delta^{N}}\sum_{n=1}^{N}%
\delta^{n-1}1_{\{a_n=a\}},\mbox{  } a\in A.
\]

The sequence $\vec{a}^N(x^\star_{k,\mu_1})$ will  be the sequence of actions required from player 1 in phase 2.2,
when player 1 reports $k\in L_S$ and
player 2 sends the message $\mu_1\in M_S\cup\{\square\}$. For $\mu_1=\square$, we let $\vec{a}^N(x^\star_{k,\mu_1})$ be an arbitrary sequence
of actions, that does not depend on $k\in L_S$.

Similarly, $\vec{b}^N(y) \in B^N$ is a vector that approximates the mixed action $y$ in terms of discounted frequencies.

We set $K_1:=\max\{|L_S|,|M_T|\}$, and we let $\alpha_1:L_S\to A^{K_1}$ and $\beta_1:M_T\to B^{K_1}$ be arbitrary one-to-one maps. Similarly, we set
$K_2:=1+\max\{|L_T|,|M_S|\}$, and we let $\alpha_2:L_T\cup\{\square\}\to A^{K_2}$ and $\beta_2:M_S\cup\{\square\}\to B^{K_2}$ be arbitrary one-to-one maps. The maps $\alpha_1$ and
$\beta_1$ are used to encode reports on one's own state into sequences of actions, while the maps $\alpha_2$ and $\beta_2$ are used to encode messages on the other player's state into
sequences of actions.

We let $\pi^1\in \overset{\circ}{\Delta}(L_T)$ and $\pi^2\in
\overset{\circ}{\Delta}(M_S)$ be arbitrary distributions with full
support.
\bigskip

We now proceed to the definition of a strategy profile $(\sigma_\delta,\tau_\delta)$. The definition involves additional parameters $\theta,
\zeta ,$ and $\psi^i,\psi^i_\square$ ($i=1,2$), all in $(0,1)$, which will be chosen later. We first define the
profile only at information sets that are not ruled out by the definition of $(\sigma_\delta,\tau_\delta)$ at earlier information sets. The
definition of $(\sigma_\delta,\tau_\delta)$ at information
sets that are reached with probability zero will be provided after.

\begin{description}
\item[Phase 1] It lasts $K_1$ stages. Player 1 plays the sequence $\alpha_1(\textbf{l}_S)$ of actions, and player 2 plays the sequence
$\beta_1(\textbf{m}_T)$ of actions.
\item[Phase 2] It is divided into two subphases, \textbf{Phase 2.1} and \textbf{Phase 2.2}.
\begin{description}
\item[Phase 2.1] It lasts $K_2$ stages. Player 1 first draws a message $\lambda_1\in L_T\cup \{\square\}$. The probability assigned to $\square$,
({resp.} to each $l'_T\in L_T$),  is equal  to $1-\zeta$ ({resp.}  $\zeta\times \ep_1(l'_T|\textbf{l}_T)$). Symmetrically, player 2
draws a message $\mu_1\in M_S\cup\{\square\}$. The probability assigned to $\square$, ({resp.} to each $m'_S\in M_S$) is equal to $1-\zeta$,
({resp.}   $\zeta\times \ep_2(m'_S|\textbf{m}_S)$).

In that phase, the players play the sequences $\alpha_2(\lambda_1)$ and $\beta_2(\mu_1)$ of actions.
\item[Phase 2.2] It lasts $\nu := \lfloor \frac{\ln (1-\theta)}{\ln\delta}\rfloor$ stages. Player 1 infers $\mu_1$ from the actions
played by player 2 in Phase 2.1, and plays the sequence $\vec{a}^\nu(x^\star_{\textbf{l}_S,\mu_1})$ of actions. Meanwhile, player 2
infers $\lambda_1$ from the actions played by player 1 in Phase 2.1, and plays the sequence
$\vec{b}(y^\star_{\textbf{m}_T,\lambda_1})$ of actions.
\end{description}
\item[Phase 3] It lasts $K_2$ stages. Player 1 draws a message $\lambda_2\in L_T$. The distribution of $\lambda_2$ depends on $\lambda_1$.
If $\lambda_1=\square$, the probability assigned to $\textbf{l}_T$ ({resp.} to each $l'_T\neq \textbf{l}_T$), is equal
 $(1-\psi^1_\square)+\psi^1_\square \times\pi^1(\textbf{l}_T)$ ({resp.}  $\psi_\square^1\times \pi^1(l'_T)$). If
 $\lambda_1\neq \square$, the probability assigned to $\lambda_2$ is equal to
 $(1-\psi^1)+\psi^1\times \pi^1(\lambda_2)$ if $\lambda_2=\textbf{l}_T$, and it is equal $\psi^1\times \pi^1(\lambda_2)$
otherwise.  Player 2 draws a message $\mu_2\in M_S$. The distribution of $\mu_2$ depends on $\mu_1$, and is obtained as for player 1.

In this phase, the players play the sequences $\alpha_2(\lambda_2)$ and $\beta_2(\mu_2)$ of actions.
\item[Phase 4] It contains all remaining stages. We denote by $N = K_1 + 2K_2 + \nu+1$ its first stage. Let $h=(a_n(h),b_n(h))_{n< N}\in (A\times B)^{N-1}$
be the history of moves up to stage $N$. Player 2 infers from $h$ the belief $p_n(h)$ held by player 1
in each stage $n<N$ along $h$. In this computation, the report of player 1 in \textbf{Phase 1} is assumed to be truthful. For $n<N$,
the belief $q_n(h)\in \Delta(T\times L_T)$ is defined in a symmetric way.
The players compute
    \[c_1(h)=\delta^{-N}\sum_n(1-\delta)\delta^{n-1}c(p_n(h),a_n(h))\mbox{ and } c_2(h)=\delta^{-N}\sum_n(1-\delta)\delta^{n-1}c(q_n(h),b_n(h)),\]
    where the sum is taken over all stages $n$ of \textbf{Phases 1}, \textbf{2.1} and \textbf{3}. Players then start playing according to the
    equilibrium profile of the semi-ignorant game $\Gamma(p_n(h),q_n(h))$, with payoff $(u_\star(p_n(h))+c_1(h),v_\star(q_n(h))+c_2(h))$.
\end{description}

Some interpretation may be helpful. In Phase 2.1, the message $\square$ is uninformative,\footnote{Since its probability does not
 depend on signals.} and is sent with high probability. In Phase 3, the level noise in the message sent by player 1 depend on player 1's first message, and is either $\psi^1$ if the first message was informative, or $\psi^1_\square$ otherwise.

\subsection{Equilibrium Strategies -- Parameter values}

We now fix the parameter values, starting with $\theta$. As $\delta\to 1$, the discounted weight of the $\lfloor \frac{\ln(1-\theta)}{\ln\delta}\rfloor$
stages of Phase 2.2 converges to $\theta$. Thus, $\theta$ is a measure of the contribution of the checking phase 2.2 to the total payoff. We choose
$\theta\in (0,1)$ to be small enough so that the following set of inequalities is satisfied:
\begin{eqnarray}\label{cond1}
(1-\theta)\E[u_\star(p_{l_S,\textbf{m}_S})|\textbf{l}_S=l_S,\mu_1=m_S]&>&
u_\star(p(\cdot|\textbf{l}_S=l_S,\mu_1=m_S)),\ \forall m_S\in M_S\\
 (1-\theta) u_{\star\star} &>&\gamma^1,\label{cond2}
\end{eqnarray}
together with the symmetric conditions for player 2.

By construction, the conditional distribution of $\textbf{m}_S$ given $(\textbf{l}_S,\mu_1)=(l_S,m_S)$
is independent of $\zeta$, and only depends on the fixed map $\ep_2$. Since $\ep_2(\cdot|m'_S)$ has full support for each $m'_S$, this conditional
distribution has full support. Therefore, the residual information held by player 2
is still valuable to player 1, whatever be $\mu^1\in \{\square\}\cup M_S$. In particular, (\ref{cond1}) holds with $\theta=0$, and thus also for
$\theta>0$ small enough. Because $\gamma^1<u_{\star\star}$,
condition (\ref{cond2}) is also satisfied for small $\theta$.

Condition (\ref{cond1})
ensures that, even if payoffs in phase 2.2 are very low, the weight $\theta$ of phase 2.2 is so small, that the residual value of the
information held by
player 2 can still offset the cost incurred when playing the prescribed sequence in phase 2.2. Condition (\ref{cond1}) is designed to make
sure that, when in phase 2.2, player 1 will rather play the prescribed sequence of actions, than switch to an optimal action.

Observe that with probability $1-\zeta$, player 1 receives no information prior to phase 3. Hence, for $\zeta>0$ small, the
bulk of information exchange takes place in phase 3. Condition (\ref{cond2}) ensures that, even if all information exchange is
postponed to phase 3, payoffs as high as $\gamma^1$ can be implemented.

\bigskip
Choose $\zeta\in (0,1)$ to be small enough so that the two inequalities
\begin{equation}\label{cond3}(1-\zeta) u_\star +\zeta u_{\star\star}<\gamma^1 <(1-\zeta)(1-\theta) u_{\star\star}
\end{equation}
hold, together with the analog inequalities for player 2.

In phase 2.2, the (conditional) optimal payoff of player 1 is $u_\star$ if $\mu_1=\square$, and does not exceed $u_{\star\star}$ if
$\mu_1\neq \square$.
The first inequality ensures that the probability $1-\zeta$ of \emph{not} disclosing
information in phase 2.1 ($\mu_1=\square$) is so high that the expectation of the optimal payoff given $\mu_1$ does not exceed $\gamma^1$. That is,
additional information must be disclosed in phase 3 in order to implement $\gamma$. This inequality, together with (\ref{cond2}), will allow us
to adjust other parameter values in a way that the overall payoff is $\gamma$.
The second inequality in (\ref{cond3}) does not play a critical role.

\bigskip
We now choose the value of $\psi^2\in (0,1)$ small enough so that, for every $l_S\in L_S,m_S\in M_S$,
\begin{equation}\label{cond4}
(1-\theta)\E[u_\star(p(\cdot|\textbf{l}_S,\mu_1,\mu_2))|\textbf{l}_S=l_S,\mu_1=m_S]>u_\star(p(\cdot|\textbf{l}_S=l_S,
\mu_1=m_S)).
\end{equation}
In this expression, $p(\cdot|\textbf{l}_S,\mu_1,\mu_2))$ is the belief held by player 1 at the beginning of phase 4, after having received the two
messages $\mu_1,\mu_2$ of player 2. The left-hand side of (\ref{cond4}) is continuous w.r.t. $\psi^2$.
For $\psi^2=0$, $\mu_2$ is equal to $\textbf{m}_S$ with probability 1, and (\ref{cond4}) therefore holds by (\ref{cond1}). Hence (\ref{cond4}) holds for $\psi^2>0$
small enough.

Observe that all parameters values $\zeta, \theta,\psi^2$ are independent of the discount factor. The last parameter, $\psi^2_{\square}$ is chosen such that
 the expected payoff of player 1 is $\gamma^1$.
 We first argue that for a given $\psi^2_{\square}$, the limit discounted payoff of player 1, as $\delta \to 1$, is equal to%
\footnote{We here abuse notation, since $N\to +\infty$ as $\delta\to 1$. However, the limit of $\E[u_\star(\textbf{p}_N)]$ is well-defined.}
 \begin{equation}\label{cond5}
 \theta\E[u(p(\cdot|\textbf{l}_S,\mu_1),x^\star_{\textbf{l}_S,\mu_1})]+(1-\theta) \E[u_\star(\textbf{p}_N)].
 \end{equation}

Here is why. The contribution of Phases 1, 2.1 and 3 vanishes, as the length of these phases is fixed
 independently of $\delta$. The expected
 payoff in phase 2.2 converges\footnote{Because the approximation of $x^\star$ by $x_\delta(\vec{a}(x^\star))$ becomes perfectly accurate
 as $\delta \to 1$.}
 to $\E[u(p(\cdot|\textbf{l}_S,\mu_1),x^\star_{\textbf{l}_S,\mu_1})]$. Finally, for a fixed $\delta$, the expected continuation payoff
 from stage $N$ is equal to $\E[u_\star(\textbf{p}_N)+c_1(\textbf{h}_N)]$. As Lemma \ref{lemm_cost} will show,
 $\E[c_1(\textbf{h}_N)]$ will converge to 0.

 Observe that for $\psi^2_{\square}=0$, and following $\mu_1=\square$, the message $\mu_2$ of player 2 is non-informative. Thus,
 conditional on the event that
 $\mu_1=\square$, player 2 does not disclose information prior to phase 4.
Thus, for $\psi^2_{\square}=0$, the left-hand side of (\ref{cond5}) does not exceed $(1-\zeta)u_\star +\zeta u_{\star\star}$ which by (\ref{cond3}) is less than
$\gamma^1$. If $\psi^2_{\square}=1$, following $\mu_1=\square$ the message $\mu_2$ is fully informative, and the
left-hand side of (\ref{cond5}) is at least equal to $\zeta u_\star +(1-\zeta)\left(\theta u_\star +(1-\theta)u_{\star\star}\right)$, which
exceeds $\gamma^1$ by (\ref{cond3}). It follows that for $\delta$ high enough, say $\delta \geq \bar\delta _1$, there exists
$\psi^2_{\square}(\delta)\in (0,1)$, such that the discounted payoff of player 1 is equal to $\gamma^1$, and such that
$\psi^2_{\square}(1):=\lim_{\delta\to 1}\psi^2_{\square}(\delta)\in (0,1)$.

\bigskip
We conclude this section by discussing how high should $\delta$ be, for the profile $(\sigma_\delta,\tau_\delta)$ to be well-defined, and by discussing
beliefs and actions off-equilibrium.

We first argue that the costs $c_1(h)$ and $c_2(h)$ are small.
\begin{lemma}\label{lemm_cost} There is $c>0$ such that for every $\delta\geq \bar\delta_1$ and every
$h\in H_N$, one has
\[c_1(h)\leq (1-\delta)c.\]
\end{lemma}

\begin{proof}
Because payoffs are bounded by 1, one has
\begin{eqnarray*}c_1(h)&\leq &(K_1+2K_2)(1-\delta)\delta^{-N}
= (K_1+2K_2)\frac{(1-\delta)}{\delta^{K_1+2K_2}}\delta^{-\lfloor \frac{\ln(1-\theta)}{\ln\delta}}\rfloor\\
 &\leq&  (K_1+2K_2)\frac{(1-\delta)}{\delta^{K_1+2K_2+1}}\delta^{\frac{\ln(1-\theta)}{\ln\delta}}
=  (K_1+2K_2)\frac{(1-\delta)}{\delta^{K_1+2K_2+1}}\frac{1}{\ln (1-\theta)},
\end{eqnarray*}
and the result follows.
\end{proof}
\bigskip

For $\delta\leq 1$ (including $\delta=1$), denote by $\calP(\delta)$ the support of $\textbf{p}_N$ when $\psi^2_{\square}$ is set to
$\psi^2_{\square}(\delta)$,
and define $\calQ(\delta)$ in a symmetric way.
Since $\pi_2$ and $\ep_2(\cdot|m_S)$ have full support, and since $\psi^2,\psi^2_{\square}(1)\in (0,1)$,
one has $\textbf{p}_N\in
\overset{\circ}{\Delta}  ^\dag_{\textbf{l}_S}(S\times M_S)$, with probability 1.

Because $\calP(1)$ and $\calQ(1)$ are finite sets, and by Proposition \ref{prop_1}, there is $\bar\delta_2<1$, $\ep>0$, and neighborhoods
$V(p)$ of $p\in \calP$, $V(q)$ of $q\in \calQ$, such that any payoff in
$[u_\star(p'),u_\star(p')+\ep]\times
[v_\star(q'),v_\star(q')+\ep]$ is a sequential equilibrium of $\Gamma(p',q')$, for every $p\in \calP,p'\in V(p)$, and $q\in \calQ,q'\in V(q)$.

In addition, we choose the neighborhoods $V(p)$, $V(q)$ to be small enough, and $C>0$ so that the conclusion of Proposition \ref{prop_2} holds for
every $p\in \calP,p'\in V(p)$, and $q\in \calQ,q'\in V(q)$.

We choose $\bar\delta_3<1$ to be high enough so that the following conditions are met for each $\delta\geq \bar\delta_3$: (i)
every  $p'\in \calP(\delta)$
belongs to  $V(p)$ for some $p\in \calP$; (ii) $(1-\delta)c\leq \ep$.

For $\delta\geq \bar\delta_3$, the profile $(\sigma_\delta,\tau_\delta)$ is then well-defined,
at any information set that is not ruled
out by the definition of $(\sigma_\delta,\tau_\delta)$ at earlier stages.

\bigskip

Consider now an information set $I^1_{l,h}$ that is reached with probability 0,
and assume that the information set $I^1_{l,h'}$
is reached with positive probability, where $h'$ is the longest prefix of $h$.

If the sequence $h$ of actions has probability zero,
then we let beliefs at $I^1_{l,h}$ and at all subsequent information
sets coincide with the belief held at $I^1_{l,h'}$. Player 1 repeats the action that is optimal at $I^1_{l,h'}$.

Assume now that the sequence $h$ has positive probability. This
corresponds to the case where player 1 misreported in Phase 1, and
played consistently with his report afterwards. Then the belief of
player 1 at $I^1_{l,h}$ is well-defined by Bayes' rule (and is
independent of player 1's strategy), and only assigns a positive
probability to information sets $I^2_{m,h}$ that are reached with
positive probability under $\tau_\star$. We let $\sigma_\star $
play at $I^1_{l,h}$ a best reply to $\tau_\star$.

By construction, sequential
rationality holds at any information set $I^1_{l,h}$ that is reached with probability zero.
One can
verify that beliefs are consistent
with $(\sigma_\star,\tau_\star)$. We omit the proof.

\subsection{Equilibrium properties}

We claim that the profile $(\sigma_\delta,\tau_\delta)$ is a sequential equilibrium profile for $\delta<1 $ high enough.

Let $\eta>0$ be small enough so that
\[\E_{\ep_2,x^\star}[l_S\to k]<\E_{\ep_2,x^\star}[l_S\to l_S]-2\eta\mbox{ for every }l_S, k\in L_S,l_S\neq k,\]
and we choose $\bar\delta_4<1$ such that
\[\E_{\ep_2,x_\delta(\vec{a}(x^\star))}[l_S\to k]<\E_{\ep_2,x_\delta(\vec{a}(x^\star))}[l_S\to l_S]-\eta\mbox{ for every }l_S, k\in L_S,l_S\neq k\mbox{ and }\delta\geq \bar{\delta}_4.
\]
We finally choose $\delta_5<1$ to be such that
$1-\delta^{K_1+2K_2}+(1-\delta)C<\eta\delta^{K_1+2K_2}$ for each $\delta\geq \bar{\delta}_5$.

We now verify that $(\sigma_\star,\tau_\star)$ is a sequential equilibrium, as soon as $\delta\geq\max\{\bar\delta_4,\bar\delta_5\}$. It is sufficient
to check that sequential rationality holds at any information set that is reached with positive probability.
Let such an information set $I_{l,h}$ be given,  and let $n$ be the stage to which  $I_{l,h}$ belongs.
If stage $n$ belongs to phase 4, then sequential rationality at $I_{l,h}$ follows because continuation strategies in phase 4 form a sequential equilibrium
of the associated self-ignorant game.
Assume then that $n < N$.

We will make use of the following observation that holds because $\ep_1(\cdot)$, $\ep_2(\cdot)$, $\pi^1$ and $\pi^2$ have full support: if $I_{l_S,l_T,h}$ is reached with positive probability, then the set of actions
that are played with positive probability at $I_{l_S,l_T,h}$ does not depend on $l_T$, and, therefore,
the information set $I_{l_S,l'_T,h}$ is also reached with positive probability, for every $l'_S\in L_S$.
We note that the compensation made in phase 4 implies that player 1 is indifferent at $I_{l_S,l_T,h}$ between all actions that are played with positive probability.
One thus simply needs to check that player 1 cannot increase his continuation payoff by playing some other action, $a$.

Assume first that $n$ belongs to either phase 2.1, 2.2 or to phase 3. In that case, the set of actions that are played at $I_{l,h}$ does
not depend on $l$. Hence, when playing $a$, player 1 triggers a myopic play by player 2, and player 1's overall payoff in that case does not
exceed
\[(1-\delta)u_\star(p_n)+\delta \E[u_\star(\textbf{p}_{n+1})|l,h].\]
On the other hand, the expected continuation payoff of player 1 at $I_{l,h}$ is at least $\delta^N\E[u_\star(\textbf{p} _N)|l,h]$. Sequential rationality
then follow from the choice of parameters.

Assume finally that stage $n$ belongs to phase 1. Again, it is not profitable to
switch to an action that triggers a myopic play from player 2. What if player 1, instead of reporting $l_S$,
chooses to report $k\neq l_S$ ?
Then, as above, the choice of parameters ensures that it is optimal for player 1 to play consistently with $k$,
at least until phase 4. Such a deviation yields a payoff (discounted back to $h$) of at most
\[\delta^{-n}\left(\delta^{K_1+K_2}\E_{\ep_2,x_\delta(\vec{a}(x^\star))}[l_S\to k]+(1-\delta)\left(1+\cdots +\delta^{K_1+2K_2-1}\right)
+\delta^N\E[u_\star(\textbf{p}_N)+(1-\delta)C]\right).\]
On the other hand, player 1's continuation payoff when reporting truthfully is at least
\[\delta^{-n}\left(\delta^{K_1+K_2}\E_{\ep_2,x_\delta(\vec{a}(x^\star))}[l_S\to l_S]+\delta^N\E[u_\star(\textbf{p}_N)+(1-\delta)C]\right).\]
We stress that the distribution of $\textbf{p}_N$ is the same in both expressions, because the distribution of $(\mu_1,\mu_2)$
does not depend on player 1's report. The result follows, by the choice of $\bar\delta_4$ and $\bar\delta_5$.
\end{document}